\documentclass[11pt,leqno]{article}
\usepackage[english]{babel}
\usepackage{tikz}
\usepackage{enumerate}
\usepackage{graphicx}
\usepackage{epstopdf}
\usepackage{booktabs,tabularx}
\usepackage{verbatim}
\usepackage{etoolbox}
\usepackage{tkz-euclide}
\usepackage{geometry}
\usepackage{mathtools}
\usepackage{color}
\usepackage[title]{appendix}
\usepackage[all]{xy}
\usepackage{tikz-cd}
\usepackage{blindtext}
\usepackage[T1]{fontenc}
\usepackage{amsthm}
\usepackage{amsfonts}
\usepackage{txfonts}
\usepackage{palatino,amssymb,epsfig}
\usepackage{latexsym,epsf,epic,amscd}
\usepackage{mathrsfs}
\usepackage{enumitem}
\usepackage{caption}
\usepackage{subcaption}
\usepackage{xcolor}
\usepackage{ulem}
\usepackage{multirow}

\usepackage{CJKutf8}

\usepackage{appendix}
\usepackage{amsmath}
\usepackage{bookmark} 
\usepackage[nottoc]{tocbibind}
\hypersetup{
colorlinks=true,
linkcolor=black,
citecolor=black,
anchorcolor=black,
urlcolor=black}
\allowdisplaybreaks[4]
\numberwithin{equation}{section}
\DeclareFontFamily{OMX}{yhex}{}
\DeclareFontShape{OMX}{yhex}{m}{n}{<->yhcmex10}{}
\DeclareSymbolFont{yhlargesymbols}{OMX}{yhex}{m}{n}
\DeclareMathAccent{\wideparen}{\mathord}{yhlargesymbols}{"F3}

\title{Convergences of Combinatorial Ricci Flows to  Degenerated Circle Packings in Hyperbolic Background Geometry}
\author{Guangming Hu, Sicheng Lu, Dong Tan, Youliang Zhong and Puchun Zhou}

\date{}
\newtheorem{theorem}{Theorem}[section]
\newtheorem{lemma}[theorem]{Lemma}

\newtheorem{proposition}[theorem]{Proposition}
\newtheorem{corollary}[theorem]{Corollary}

\theoremstyle{definition}
\newtheorem{definition}[theorem]{Definition}
\newtheorem{remark}[theorem]{Remark}

\renewcommand{\arraystretch}{1.5}

\newcommand{\pp}[2]{\frac{\partial#1}{\partial#2}}

\newcommand{\cRf}{combinatorial Ricci flow}
\newcommand{\Lrz}{\mathbb{E}^{2,1}}

\renewcommand{\arraystretch}{1}

\newcommand{\DEF}{\emph}

\newcommand{\PRO}{\textcolor{brown}} 

\begin{document}
\maketitle 

\begin{abstract} 
This paper investigates a kind of degenerated circle packings in hyperbolic background geometry. A main problem is whether a prescribed total geodesic curvature data can be realized by a degenerated circle packing or not. We fully characterize the sufficient and necessary conditions and show the uniqueness. Furthermore, we introduce the combinatoral Ricci flow to find the desired degenerated circle packed surface, analougus to the methods of Chow-Luo \cite{chow-Luo} and Takatsu \cite{Takatsu}. 

\end{abstract}
 

\section{Introduction}
\subsection{Background}
In the last few  decades, the research of discrete conformal structures on  polyhedral surfaces has been widely concerned, which is a discrete simulation corresponding to conformal structures on  smooth surfaces. The circle packings on  polyhedral surfaces can be regarded as a  naturally  discrete problem of conformal structures. 
In 1936, Koebe  \cite{Koebe} showed the projective rigidity of circle packings on the Riemann sphere. In 1970,  Andreev \cite{andreev1} rediscovered this rigidity theorem. In 1976, Thurston \cite{thurston} investigated circle packings on both Euclidean and hyperbolic closed surfaces as a method to explore hyperbolic 3-manifolds. 

The variational principle and \cRf\ play an impormtant role in the study of circle packing. Colin de Verdiere \cite{colin} utilized the variational principle to reconfirm the existence and uniqueness of Thurston's circle packing theorem. 
This principle inspires a wealth of remarkable works on polyhedron surfaces and related topics, among which are \cite{bo, Leibon, Rivin, Luo12, Luo13}.
However, the proof of existence part is non-constructive. Chow-Luo introduced the \cRf\ in \cite{chow-Luo} to control the evolution of geodesic curvature, giving a new proof of the existence. This is an analogue of Hamilton's Ricci flow on surfaces \cite{Hamilton2}. Their result can be regarded as an algorithm for finding the desired circle packing that converges exponentially fast. The sufficient and necessary conditions for convergence are also given.

The settings of variational principle and \cRf\ in Euclidean and hyperbolic geometry can not be applied directly to spherical geometry, due to the breaking of convexity \cite{Luo}. Recently, Nie \cite{nie} discovered a new convex functional where geodesic curvature is replaced by the total geodesic curvature, and obtains the rigidity of circle pattern in spherical background geometry. The fifth author and his collaborators introduce the \cRf\ in \cite{GHZ} and obtains convergence results similar to Chow-Luo. The first author and his collaborators extend the study of total geodesic curvature to hyperbolic geometry \cite{BHS}. The circles in the packing are allowed to be usual circles, horocycles and hypercycles. In general, the sufficient and necessary condition for the existence and rigidity of circle packing or pattern and the convergence of \cRf\ in various settings are established.

Now it is a natural question to ask: what is the asymptotic behavior of the \cRf\ when the sufficient and necessary conditions are not valid and take critical values; can we describe it geometrically? This problem is firstly proposed by Chow-Luo in \cite{chow-Luo}. Takatsu \cite{Takatsu} addresses this problem in the original setting of Chow-Luo, based on an infinitesimal description of degenerated circle patterns. 
In this paper, we try to answer a similar question under the setting of \cite{BHS} and give a geometric interpretation of degenerated generalized  circle packed surfaces. These degenerated structures provide a partial boundary of the space of generalized  circle packed surfaces. And a convergence result is obtained, which allows the \cRf\ converge to a boundary point. Similar results on circle patterns in spherical background geometry have been obtained by partial authors of this paper in \cite{HLSZ}.

To state our main results, we first clarify some basic settings.

\subsection{Generalized hyperbolic circle packing and realization problem}
Let $S$ be a connected closed surface and $G$ be a simple triangulation of $S$. 
Denote by  $V$, $E$, $F$  the sets of vertices, edges and triangles of $G$, respectively.
For simplicity, we usually identify $V$ as the set $\{1,2,\cdots,\lvert V \rvert\}.$

A generalized circle packed surface is obtained by gluing generalized hyperbolic triangles along their edges according to the triangulation $G$, such that the circle packing on adjacent triangles matches (See Definition \ref{def:three_circle}, \ref{def:circ.pack.surf.}). The triangulation $G$ is called the nerve of this circle packed surface. 
A surface is completely determined by the geodesic curvatures of the circles centered at the vertices, which can be encoded as a vector $\vec{k}:=(k_v)_{v\in V} \in \mathbb{R}^{V}$. Hence the space of all generalized circle packed surfaces with nerve $G$ is identified with the first quadrant $\Omega_k := \mathbb{R}_{>0}^{V} $.

Finding a structure with prescribed geometric quantities is a key problem in the study of polyhedral surfaces. Analogous to the classical setting of prescribed Gaussian curvatures at vertices,    the problem of prescribed total geodesic curvatures of vertices in hyperbolic background geometry is studied in \cite{BHS}.

For any non-empty subset $I \subset V$, let $F_I \subset F$ be the set of triangles with at least one vertex in $I$, and define a function $\eta^I$ on $\mathbb{R}_{\geq 0}^V$ as 
\begin{equation}\label{def:eta}
    \eta^I(\vec{T}):=  \pi \big\lvert F_I \big\rvert - \sum_{ i\in I} T_i 
\end{equation}
for any $\vec{T}:=(T_v)_{v\in V} \in \mathbb{R}_{\geq 0}^V$. 
Denote by $\Omega_T$ the bounded open polyhedron
\begin{equation}\label{Omega_T}
\Omega_T :=\left\{\ \vec{T} \in\mathbb{R}^V_{>0} \ \left\vert\ \eta^I(\vec{T})>0,\ \forall\ \textrm{non-empty subset }\ I\subset V \right.\ \right\} \ .
\end{equation}
It is shown in \cite{BHS} that there exists a generalized circle packed surface realizing the prescribed total geodesic curvature $\vec{T} \in \mathbb{R}_{>0}^V$ if and only if $\vec{T} \in \Omega_T$. This result is obtained by showing that the map  
\begin{equation}\label{map:Phi_T}
\begin{split}
    \Phi: \Omega_k &\to \mathbb{R}_{>0}^V \\
    \vec{k} &\mapsto (T_v(\vec{k}))_{v\in V}
\end{split}
\end{equation}
recording the total geodesic curvature of circles is actually a homeomorphism onto $\Omega_T$ (Also see Theorem \ref{mr}). Furthermore, the combinatorial Ricci flow is introduced to find the desired surface with the prescribed total geodesic curvature.

Since $\Omega_T$ admits an apparent boundary in Euclidean space, one may naturally ask the following problem:

\noindent{\bf Problem }\; {\em Can the point on the boundary $\partial\Omega_T$ be realized as the prescribed total geodesic curvature of some geometric surface? How to find it?}

In this paper, we will give a partially affirmative answer to the problem.

Denote by $\overline{\Omega}_k := \mathbb{R}_{\geq0}^V$ and 
\begin{equation}\label{def:widehat_T}
\widehat{\Omega}_T:=\left\{\ \vec{T} \in\mathbb{R}^{ V}_{\geq 0}\ \left\vert\ \eta^I(\vec{T})>0,\ \forall\ \textrm{non-empty }\ I\subset V \right. \right\} .
\end{equation}
Let 
\begin{equation*}
 \partial \Omega_k := \overline{\Omega}_k \setminus \Omega_k
 \quad \text{and} \quad 
 \partial_0 \Omega_T := \widehat{\Omega}_T \setminus \Omega_T
\end{equation*}
where the subscript $0$ emphasizes the boundary $\partial_0 \Omega_T$ is a partial boundary rather that the whole boundary of $\Omega_T$. 
A point in $\partial \Omega_k$ corresponds to the unique degenerated generalized circle packed surface (See Proposition \ref{prop:deg.circ.pack.surf.}), on which the total geodesic curvature of packed circles are still well-defined. It is the geometric object for our realization problem.

\begin{theorem}\label{thm:realization}
    For any $\vec{T} \in \partial_0 \Omega_T$, there exists a unique degenerated generalized circle packed surface realizing the prescribed total geodesic curvature $\vec{T}$. 
\end{theorem}

Similar to Theorem \ref{mr}, Theorem \ref{thm:realization} is obtained by showing that the map recording the total geodesic curvature is a homeomorphism. 
This map admits a layered structure. For any non-empty $I \subset V$, the \DEF{stratum} $\partial^I \Omega_k$ (or $\partial_0^I \Omega_T$ resp.) of the boundary $\partial \Omega_k$ (or $\partial_0 \Omega_T$ resp.) indexed by $I$ is the subset of elements satisfying
\begin{equation}\label{def:stratum}
    k_i = 0 \ (\text{or } T_i = 0 \text{ resp.}) \iff i \in I .
\end{equation}
By denoting
\begin{equation*}
    \partial^\varnothing \Omega_k = \Omega_k
    \quad \text{and} \quad
    \partial_0^\varnothing \Omega_T = \Omega_T ,
\end{equation*}
we have the following decomposition
\begin{equation}
    \overline{\Omega}_k = \bigcup_{I \subset V} \partial^I \Omega_k
    \quad \text{and} \quad
    \widehat{\Omega}_T = \bigcup_{I \subset V} \partial_0^I \Omega_T .
\end{equation}
Theorem \ref{thm:realization} is then derived from the following result on continuous extension of $\Phi$. 
\begin{theorem}\label{Main_result_1}
The homeomorphism $\Phi : \Omega_k \to \Omega_T$ extends continuously to a homeomorphism $\widehat{\Phi} : \overline{\Omega}_k \to \widehat{\Omega}_T$, with $\widehat{\Phi}(\partial^I\Omega_k) = \partial_0^I \Omega_T $ for any $I\subset V$.
\end{theorem}
Note that by our convention, the result of \cite{BHS} for $\Phi$ is the case where $I=\varnothing$.

\begin{remark}
  Since the degenerated surfaces allow some hypercycles to become geodesics, they may be disconnected (see Remark \ref{rmk:disconnect}). These surfaces are also closed to \emph{crowned surfaces}, which arise as the complement of a general geodesic measured lamination on a hyperbolic surface. See \cite{CB, G19} for more details about crowned surfaces.   
\end{remark} 

\subsection{Combinatorial Ricci flow}
To find the actual degenerated generalized circle packed surface with prescribed total geodesic curvature $\vec{T}\in \partial_0 \Omega_T$, we take advantage of the combinatorial Ricci flow, which is similar to the work of Chow-Luo \cite{chow-Luo}. The flow is even defined for degenerated surfaces on certain strata.

Given possibly empty $J \subset I \subset V$ and $\vec{T}= (T_v)_{v \in V}\in \partial_0^I \Omega_T$, the \DEF{\cRf}\ on $\partial^J \Omega_k$ is defined by 
\begin{equation}\label{CRF1}
    \frac{\mathrm{d} k_v(t)}{\mathrm{d} t} = -\big(T_v(\vec{k}(t))-T_{v}\big) \cdot k_v(t),\ \quad \text{for any } v\in V \setminus J . 
\end{equation}
Here $\vec{k}(t) = (k_v(t))_{v\in V\setminus J}$ is the flow ray in $\partial^J \Omega_k$ over $t\geq0$, with initial data $\vec{k}(0)$.

\begin{theorem}\label{Main_result_3} 
For any $\vec{T}\in \partial_0^I \Omega_T$, $J\subset I$ and $\vec{k}(0)\in \partial^J \Omega_k$, the flow ray $\vec{k}(t)$ of (\ref{CRF1}) converges to $\widehat{\Phi}^{-1} (\vec{T}) \in \partial^I\Omega_k $ as $t\to +\infty$. 
\end{theorem}

Note that $\partial^I\Omega_k$ (or $\partial_0^I\Omega_T$ resp.) lies on the boundary of $\partial^J\Omega_k$ (or $\partial_0^J\Omega_T$ resp.) whenever $J\subset I$. See Figure \ref{fig:flow_line} for a comprehensive illustration of the flow rays in Theorem \ref{Main_result_3}. 

It has been shown in \cite{BHS} that for any $\vec{T}\in \Omega_T =\partial^\varnothing \Omega_T$ , the flow ray of (\ref{CRF1}) through any initial data $\vec{k}(0) \in \Omega_k$ converges exponentially fast to the point $\Phi^{-1}(\vec{T})$. The following corollary of Theorem \ref{Main_result_3}, by taking $J=\varnothing$, extends the result to the partial boundary $\partial_0 \Omega_T $ of $\Omega_T $, allowing flow rays in the interior converge at the boundary.

\begin{corollary}\label{Main_result_2} 
For any $\vec{T} \in \partial_0 \Omega_T$ and $\vec{k}(0) \in \Omega_k$, the flow ray from $\vec{k}(0)$ diverges in $\Omega_k$. 
However, it converges in $\overline{\Omega}_k = \mathbb{R}_{\geq0}^{V}$ as $t$ tends to $+\infty$ and the limit point is exactly $\widehat{\Phi}^{-1} (\vec{T}) \in \partial \Omega_k$. 
\end{corollary}

\begin{figure}
    \centering
    \includegraphics[scale=0.85]{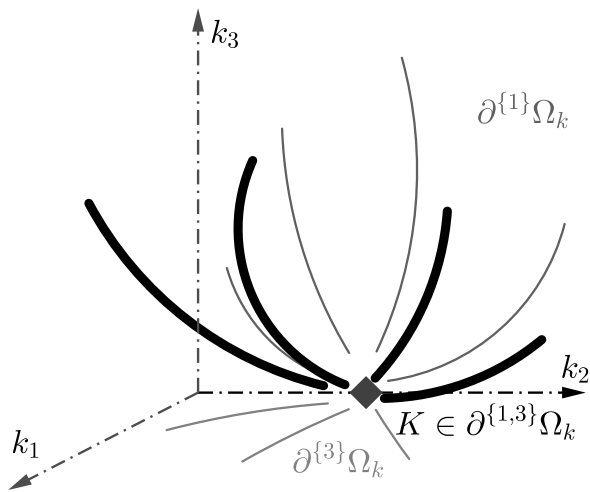}
    \caption{\small A schematic picture of flow rays. The limit point $K$, marked as a diamond, lies on $\partial^{\{1,3\}}\Omega_k$, which is a common boundary of two disjoint strata $\partial^{\{1\}}\Omega_k, \partial^{\{3\}}\Omega_k$. The thin lines lie completely in lower-dimensional strata, while the thick lines are in $\Omega_k$.}
    \label{fig:flow_line}
\end{figure}

\bigskip
\noindent\textbf{Organization} This paper is organized as follows. In Section \ref{2}, we review the settings of three-circle configuration and generalized circle packed surfaces. 
In Section \ref{3}, we introduce degenerated hyperbolic circle packed surfaces. 
In Section \ref{4}, we show the continuous extension of the homeomorphism $\Phi: \mathbb{R}_{>0}^{V} \to \Omega_T$, which finishes the proof of Theorem \ref{Main_result_1}. 
In Section \ref{5}, we prove the convergence of a combinatorial Ricci flow with prescribed total geodesic curvatures on the partial boundary of $\partial_0 \Omega_T$. 
The Appendix A provides some detailed computations on hyperbolic geometry used here, especially in Section \ref{3} and \ref{4}.

\section{The hyperbolic circle packed surfaces}\label{2}

\subsection{Three-circle configuration}\label{setting}

In this section, we draw upon some fundamental concepts on the hyperbolic geometry. Also see the Appendix for more quantitative computations.

A \DEF{generalized circle} in $\mathbb{H}^2$ is a circle of constant geodesic curvature with respect to the hyperbolic metric. 
A generalized circle is either a circle, a horocycle or a hypercycle. 
To generalise center from circles, we define the \DEF{center of a horocycle} as the unique limit point on the boundary and the \DEF{center of a hypercycle} as the geodesic connecting its two limit points on the boundary (See Figure \ref{fig:circles}).  
Using the hyperboloid and Klein model for $\mathbb{H}^2$, generalized circles have a more quantative uniform definition (See Lemma \ref{lem:gnr_circle}).

We will focus on arcs of generalized circles. The inner angle of a hypercycle arc is defined as follows:
Let $\Gamma$ be a hypercycle with geodesic axis $\gamma$ and radius $r$, and $C$ be an arc on $\Gamma$ with endpoints $c_1,c_2$. 
Let the shortest distance projection of $C$ onto the geodesic $\gamma$ be $\alpha$. Then the hyperbolic length of $\alpha$ is defined to be the (generalized) \DEF{inner angle} $\theta$ of arc $C$. See the left of Figure \ref{fig:circles}.
The relations among the radii and absolute geodesic curvature of generalized circle, and the inner angle and the length of arc of the generalized circle, are summarized in Table \ref{table:relation}. 

\begin{figure}[htbp]
    \centering
    \includegraphics[width=\textwidth]{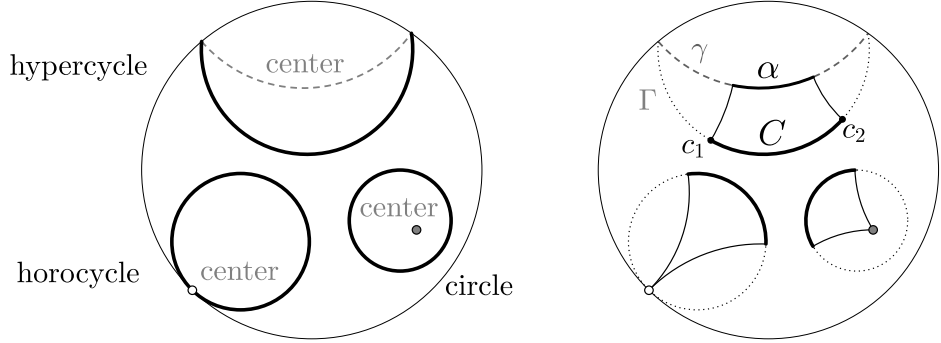}
    \caption{\small A circle, a hypercycle and a horocycle on $\mathbb{H}$.}
    \label{fig:circles}
\end{figure}

\begin{definition}[circle packed generalized triangle] \label{def:three_circle}
    Let $C_1, C_2, C_3$ be three generalized circles externally tangent to each other, arranged counterclockwise in $\mathbb{H}^2$.
    \begin{enumerate}
        \item For $i \neq j \in \{1,2,3\}$, the \DEF{edge} $e_{ij}=(v_i v_j)$ is the geodesic connecting the centers of $C_i$ and $C_j$, passing through the point tangency. When some of them is hypercycle, the edge must be perpendicular to its center geodesic.
        \item For $i \in \{1,2,3\}$, the \DEF{vertex} $v_i$ is the center of $C_i$ if $C_i$ is a circle or a horocycle. 
        If $C_i$ is a hypercycle, the vertex $v_i$ is defined as the geodesic segment between the two edges $(v_i v_j)$ and $(v_i v_k)$ where $j \neq k$ are indexes other than $i$. 
        \item A \DEF{generalized (geodesic) triangle}, denoted by $\Delta$, 
        is the closed hyperbolic polygon bounded by the three vertices $v_1, v_2, v_3$ and three edges $e_{12}, e_{23}, e_{31}$. 
        \item The three generalized circles, together with the generalized geodesic triangle, is called a \DEF{three-circle configuration}. 
        The pattern inside the closed polygon is called a \DEF{circle packed generalized triangle}.
        \item A circle packed generalized triangle is called \DEF{of curvature $(k_1, k_2, k_3)$} if the geodesic curvature of $C_i$ is constantly $k_i$ for $i = 1, 2, 3$. 
    \end{enumerate}
\end{definition}

\begin{definition}[Dual triangle]
Given a three-circle configuration with generalized triangle $\Delta$, let $t_{ij}$ be the tangent point of $C_i, C_j$ on the edge $e_{ij}=(v_i v_j)$.
The \DEF{dual triangle} $\Delta^*$ is a non-geodesic triangle inside $\Delta$, bounded by the three circles, with vertices $t_{12}, t_{23}, t_{31}$. 
\end{definition}
See Figure \ref{fig:general3circle} as an example of all the concepts involved.

\begin{table}\centering
\centering
\renewcommand\arraystretch{1.2}
\footnotesize
\begin{tabular}
{l|l|l|l}
 &  \small Radius & \renewcommand\arraystretch{0.9}\begin{tabular}{l} Absolute value of \\ geodesic curvature \end{tabular}  &  \small  Arc length \\ \hline
 \small Circle &  $0<r<\infty$&  $k=\coth r$&  $l=\theta\sinh r$  \\ 
 \small Horocycle &  $r=\infty$&  $k=1$&  -----  \\
 \small Hypercycle &  $0<r<\infty$& $k=\tanh r$ &  $l=\theta\cosh r$  \\ \hline
\end{tabular}
\caption{ \small The relationship of radius, absolute value of geodesic curvature and arc length. }
\label{table:relation}
\end{table}

\begin{lemma}\label{L23}
Given a triple $\vec{k} = (k_1, k_2, k_3) \in \mathbb{R}_{>0}^3$, there exists a unique three-circle configuration (or circle packed triangle) (up to isometry) such that the geodesic curvature of $C_i$ is exactly $k_i$ for $i = 1, 2, 3$.
\end{lemma}

\begin{proof}
    Let $C_1$ and $C_2$ be two tangent circles of curvature $k_1$ and $k_2$. 
    Denote by $t_{1,2}$ the tangent point of $C_1$ and $C_2$. 
    Let $C$ be a circle of curvature $k$ which passes $t_{1,2}$ and is perpendicular to both $C_1$ and $C_2$. 
    Take $C_3$ as the circle of curvature $k_3$ tangent to $C_1$ and perpendicular to $C$. 
    Consider the intersection number $I(k)$ of $C_2$ and $C_3$ which is an increasing function of $k$. 
    There exist $k_2 > k_1 > 1$ such that the intersection number $I(k_1) = 0$ and $I(k_2) = 2$. 
    Then there exists a unique $k_0 \in (k_1, k_2)$ such that the intersection number $I(k_0)$ is equal to $1$. 
    The lemma follows. 
\end{proof}

\begin{figure}[h]
	\centering
	\includegraphics[width=0.45\linewidth]{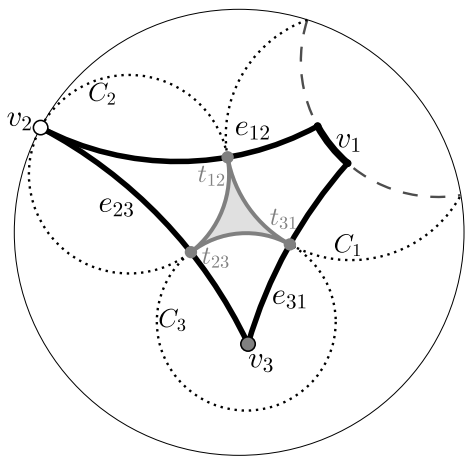}
	\caption{\small A typical three-circle configuration and a circle packed generalized triangle. The triangle is of curvature $(k_1,k_2,k_3)$ with $0<k_1<1=k_2<k_3$, hence the vertex $v_1$ is a geodesic segment while $v_2$ is an ideal vertex. The shaded region is its dual triangle.}
	\label{fig:general3circle}
\end{figure}

\subsection{Generalized circle packed hyperbolic surface}\label{setting_surface}

Let $S$ be a oriented compact surface and $G$ be a simple triangulation. 
Denote by $V$, $E$ and $F$ the set of vertices, edges and faces of $G$ respectively. 
All $f\in F$ are counterclockwise oriented, with respect to the orientation of $S$.
An edge $e\in E$ and a positively oriented triangle $f\in F$ are also represented by their vertices as $e=(ij)$ and $f=(i j k)$.

\begin{definition}\label{def:circ.pack.surf.}
Given $(S,G)$ as above, for any $\vec{k} = (k_v)_{v \in V} \in \Omega_k$ we construct a hyperbolic surface as follows: 
\begin{itemize}
    \item For any $f = (u v w) \in F (G)$, realize $f$ as a circle packed generalized triangle $\Delta_f$ of curvature $(k_u, k_v, k_w)\in \Omega_k^\Delta$. 
    \item If $f_1, f_2 \in F$ share the same edge $e=(ij)\in E$, glue the two generalized triangles $\Delta_{f_1}, \Delta_{f_2}$ along the edges corresponding to $e$ by isometry, such that the vertices of two dual triangles on these edges coincide and the tangent points of two circles on these edges coincide. 
    \item The surface obtained by gluing all edges of $\Delta_f$'s as above is called a \DEF{generalized circle packed surface} with nerve $G$ and curvature $\vec{k}$.
    \item On a generalized circle packed surface, for any $v \in V$, all circle arcs centered at vertex $v$ form a closed circle $C_v$. The collection $\{C_{v}\}_{v \in V}$ is called a \DEF{generalized circle packing} on that surface. 
\end{itemize}
\end{definition}
By the construction, all generalized circle packed surfaces with nerve $G$ are parameterized by $\Omega_k$. Hence when the surface and triangulation $(S,G)$ is given, we will simply represent a surfaces by a point $\vec{k}\in\Omega_k$.

\begin{remark}
    We omit all the adjective ``generalized'' from now on for simplicity.
\end{remark}

\begin{remark}
    \label{rmk:decomp.vert.set.}
For any $\vec{k}=(k_v)_{v\in V}\in\Omega_k$, there is a partition of $V$ given by
\[ 
V_B:=\{\ i\in V\ \lvert\ 0<k_i<1\ \},\quad 
V_P:=\{\ i\in V\ \lvert\   k_i=1\ \},\quad 
V_C:=\{\ i\in V\ \lvert\   k_i>1\ \}. 
\]
Then topologically, the circle packed surface with curvature $\vec{k}$ is obtained by removing $\lvert V_P \rvert$ points and $\lvert V_B \rvert$ open disks, all disjoint.
Geometrically, the surface contains $\lvert V_B \rvert$ closed geodesic boundaries, $\lvert V_P \rvert$ cusps and $\lvert V_C \rvert$ conical points (possibly of cone angle $2\pi$), and is smooth on the complementary. See an example in Figure \ref{fig:gnr_pack_surf}.
\end{remark}

\begin{figure}
    \centering
    \includegraphics[width=\linewidth]{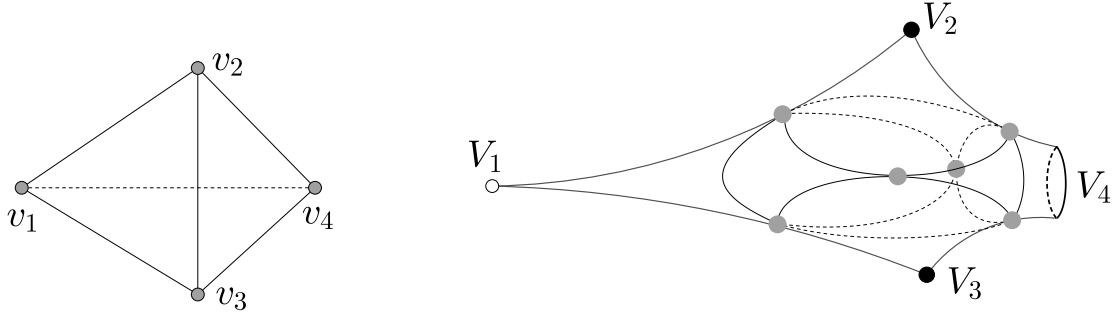}
    \caption{\small A typical generalized circle packed surface, where $G$ is the tetrahedron triangulation of the sphere. The surface has one cusp corresponding to $v_1$, a totally geodesic boundary corresponding to $v_4$, and two conical points.}
    \label{fig:gnr_pack_surf}
\end{figure}

\medskip
For $\vec{k}=(k_v)_{v\in V}\in\Omega_k$ and $f=({u} {v} {w}) \in F$, let 
$l_{v}^{f} (\vec{k})$ be the hyperbolic length of the circle arc $C_{v}^{f}$ centered at vertex $i$ in the triangle $\Delta_f$ of curvature $(k_u, k_v, k_w)$. The \DEF{total geodesic curvatures} of arc $C_{v}^{f}$ is defined as
\begin{equation}\label{def:T.g.c_v^f}
    T_{v}^{f} (\vec{k}) := l_{v}^{f} (\vec{k}) \cdot k_{v}.
\end{equation}
Recall that $F_{\{v\}}$ is the set of triangles containing $v\in V$ as one of its vertex. Then the \DEF{total geodesic curvatures} of $C_v$ on the circle packed surface with curvature $\vec{k}$ is defined as
\begin{equation}\label{def:T.g.c_v}
      T_v (\vec{k}):= \sum_{f\in F_{\{v\}}} T_v^{f} (\vec{k}).
\end{equation}

The following result tells us that all circle packed surfaces with given nerve $G$ can be reparametrized by total geodesic curvatures.
\begin{theorem}[\cite{BHS}]\label{mr}
The image of the map $\Phi$ defined by (\ref{def:T.g.c_v}) and (\ref{map:Phi_T}) is precisely $\Omega_T$ in (\ref{Omega_T}). And $\Phi:\Omega_k \to \Omega_T$
is a homeomorphism.
\end{theorem}

\section{The degenerated hyperbolic circle packed surfaces}\label{3}

To precisely define the points in $\partial\Omega_k$, we introduce generalized circles with zero geodesic curvature. We begin with three-circle configurations, allowing some of the generalized circles to be geodesics, and then glue these configurations together to form hyperbolic surfaces analogous to hyperbolic circle-packed surfaces.

\subsection{Degenerated circle packed triangle}\label{3.1}
Degenerated three-circle configurations are defined similarly to three-circle configurations in Definition \ref{def:three_circle}. 
\begin{definition}\label{def:dgn_triangle}
    A \DEF{degenerated three-circle configuration} is a three-circle configuration where some generalized circles degenerate into total geodesics of $\mathbb{H}^2$. The pattern obtained in (4) of Definition \ref{def:three_circle} is called a \DEF{degenerated circle-packed triangle}, with edges possibly degenerating to ideal points. The \DEF{dual triangle} is defined similarly, possibly having ideal vertices.
\end{definition}

\begin{lemma}\label{L123}
    Given $\vec{k}=(k_1, k_2, k_3) \in \partial \mathbb{R}_{\geq0}^3$, there exists a unique (up to isometry) degenerated three-circle configuration (or degenerated circle-packed triangle)  such that the geodesic curvature of $C_i$ is exactly $k_i$ for $i = 1, 2, 3$.
\end{lemma}

\begin{proof}
    It is sufficient to assume that $k_1 \leq k_2 \leq k_3$. 
    The case that $0 = k_1 < k_2 \leq k_3$ follows from the same argument of Lemma \ref{L23}. 

    Consider the case that $0 = k_1 = k_2 < k_3$. 
    Let $C_1$ be the geodesic joining $-1$ and $1$ in $\mathbb{H}^2$, and let $C_2$ be the geodesic joining $-1$ and $-i$ in $\mathbb{H}^2$. 
    For $s \in (-1, 1)$, let $C_3 (s)$ be the circle of curvature $k_3$ tangent to $C_1$ at $s$. 
    The intersection number $I (s)$ between $C_2$ and $C_3 (s)$ is a decreasing function of $s \in (-1, 1)$. 
    There exist $-1 < s_1 < s_2 < 1$ such that $I (s_1) = 2$ and $I (s_2) = 0$. 
    Then there exists $s_0 \in (s_1, s_2)$ such that $I (s_0) = 1$. 
    The lemma holds for $k_1 = k_2 = 0$. 

    If $k_1=k_2=k_3=0$, the degenerated circle-packed triangle is an ideal hyperbolic triangle.

\end{proof}
Lemma \ref{L123} is illustrated in Figure \ref{fig:dgn_3circle}.
\begin{figure}[hbt]
    \centering
    \includegraphics[width=\linewidth]{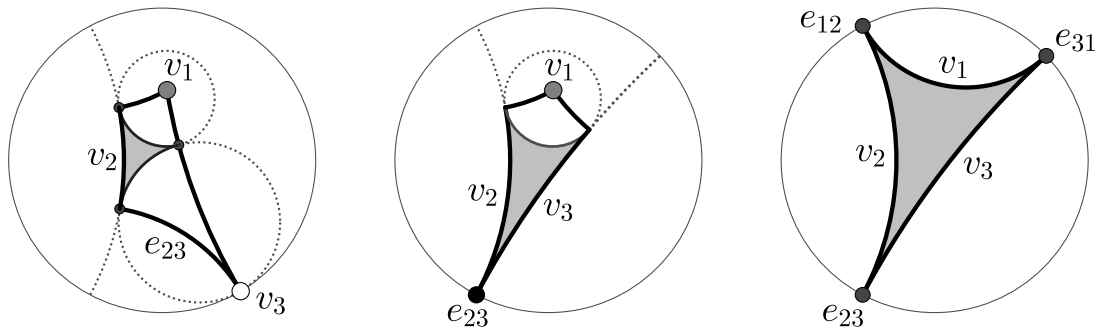}
    \caption{\small 
    Consider degenerate three-circle configurations and circle-packed triangles, where one, two, or three circles degenerate into geodesics. Note that the case where all $k_i = 0$ is distinct from the case where all $k_i = 1$. In the former, the dual triangle coincides with the generalized triangle itself.
    }
    \label{fig:dgn_3circle}
\end{figure}

\subsection{Pairing of half-edges}

To define a degenerate circle-packed surface, we introduce topological information about the triangulation $G$ to describe the gluing process along degenerate edges. 
These edges may be ideal points at the boundary of infinity.

\begin{definition}[Half-edge]\label{def:half_edge}
    Let $f = (u, v, w) \in F$ be a positively oriented triangle of $G$. 
    The oriented edge $e^f = (u, v)$ is called a \DEF{half-edge} of edge $e = (uv) \in E$ of $f$.
\end{definition}

If $g = (v, u, x) \in F$ is the other positively oriented triangle sharing the edge $e = (uv)$, then $e^g = (v, u)$ is another half-edge of $e$ of $g$. 
Each edge $e \in E$ has two half-edges in two adjacent triangles, with reversed orientations. 
Denote by $\widetilde{E}^+$ the set of all half-edges of $G$. 
There is a natural 2-to-1 map
\begin{equation}
    \label{map:edge.proj.}
    \pi_E : \widetilde{E}^+ \to E
\end{equation}
defined by forgetting its orientation.

\begin{definition}[Pairing of half-edges]
    Let $\sigma_G: \widetilde{E}^+ \to \widetilde{E}^+$ be the natural involution reversing orientations of half-edges. 
    That is, for any $(u, v) \in \tilde{E}^+$, 
    \[ \sigma_G (u,v) := (v,u). \]
\end{definition}
If $\sigma_G(e^f)$ is a half-edge inside $g\in F$, then $e$ is a common edge of $f,g$.

\subsection{Degenerated circle packed surface}
In this subsection, we generalize circle-packed surfaces (Definition \ref{def:circ.pack.surf.}) to degenerate circle-packed surfaces (Definition \ref{def:dgn_surf1}). 
Set $\vec{k} = (k_v)_{v \in V} \in \overline{\Omega}_k = \mathbb{R}_{\geq 0}^V$. 
For any $f = (u, v, w) \in F$, let $\Delta_f$ be a (degenerate) circle-packed triangle of curvatures $(k_u, k_v, k_w) \in \overline{\Omega}_k^{\Delta} := \mathbb{R}_{\geq 0}^3$.

\begin{definition}
    \label{def:spike}
    For $e = (uv) \in E$ with $k_u = k_v = 0$, let $e^f = (u, v)$ be a half-edge of edge $e$ of $f$. 
    A neighborhood around the ideal point from the degenerate edge $(uv)$ of $\Delta_f$ is called the \DEF{spike corresponding to the half-edge} $e^f$.
\end{definition}

For $e = (uv) \in E$ and $f, f' \in F(e)$, we glue $\Delta_f$ and $\Delta_{f'}$ as follows: 
\begin{itemize}
    \item If $k_u = k_v = 0$, the corresponding half-edges $\pi_E^{-1}(e) = \{e^f, e^g\}$ of $f, g$ degenerate to ideal points, and the pairing of half-edges $\sigma_G$ forms a pairing of these spikes.
    \item Otherwise, if at least one of $k_u$ or $k_v$ is not zero, we glue $\Delta_f$ and $\Delta_{f'}$ as in Definition \ref{def:circ.pack.surf.}.
\end{itemize}


Let
\[ E^0(\vec{k}) := \{\ e=(uv) \ \lvert\ k_u=k_v=0 \ \}. \  \]
The lift $\pi_E^{-1} (E^0 (\vec{k}))$ of $E^{0} (\vec{k})$ under projection $\pi_E^{-1}$ \eqref{map:edge.proj.} is a subset of $\tilde{E}^+$. 


\begin{definition}\label{def:dgn_surf1}
    For $\vec{k} \in \partial\Omega_k$, the \DEF{degenerated circle packed surface} corresponding to $\vec{k}$ is defined as a (possibly disconnected) hyperbolic surface obtained by gluing circle packed triangles as before, together with a pairing $\sigma_G$ of  $\pi_E^{-1}(E^0(\vec{k}))$. 
\end{definition}

\begin{remark}\label{rmk:disconnect}
  
  See Figure \ref{fig:dgn_packing1} for a typical example of a degenerated surface which is disconnected. 
  Since each spike is paired with another by $\sigma_G(\vec{k})$, this hyperbolic surface can be considered "virtually connected": 
  entering one spike leads to another paired spike. 
  The "tangent point at infinity" resembles a "wormhole" connecting the two triangles.

\end{remark}
A new type of boundary component arises if circle packed surface is degenerated. 
\begin{definition}
    \label{def:crown.bdy.}
    A \DEF{crowned boundary component} of a hyperbolic surface $X$ is an "unordered" sequence $(g_1, \dots, g_n)$, where $g_i$ is a complete geodesic boundary of $X$, satisfying:
    \begin{itemize}
        \item $g_n$ is tangent to $g_1$;
        \item $g_i$ is tangent to $g_{i+1}$ for $i = 1, \dots, n-1$.
    \end{itemize}
    The term "unordered" means that the indices can be cycled, i.e., any cyclic permutation of $\{1, \dots, n\}$ represents the same sequence.
\end{definition}

\begin{proposition}
    \label{prop:deg.circ.pack.surf.}
    For each $\vec{k} \in \partial \Omega_k$, it can be uniquely (up to isometry) realized as a degenerated circle packed surface as in Definition \ref{def:dgn_surf1}. 
    A boundary component of $S \setminus E^0(\vec{k})$ corresponds to a crowned boundary component of $\vec{k}$.
\end{proposition}

\begin{proof} 
    For any $f = (u, v, w) \in F$, by Lemmas \ref{L23} and \ref{L123}, there exists a unique (degenerate) circle-packed triangle (up to isometry) of curvatures $(k_u, k_v, k_w) \in \overline{\Omega}_k^{\Delta} := \mathbb{R}_{\geq 0}^3$. 
    Two spikes are paired if and only if they correspond to the same edge. 
    Combining the uniqueness of isometries for gluing (possibly degenerated) circle packed triangles (Definitions \ref{def:circ.pack.surf.} and \ref{def:dgn_surf1}), there is a unique degenerated circle packed surface realizing $\vec{k}$.

    Consider a boundary component $\beta$ of $S \setminus E^0(\vec{k})$. 
    Choose an orientation on $\beta$ such that $S \setminus E^0(\vec{k})$ is on the left of $\beta$. 
    Write $\beta$ as a sequence $(\tilde{e}_1, \cdots, \tilde{e}_n)$ of half-edges satisfying:
    \begin{itemize}
        \item The orientation of $\tilde{e}_i$ matches the chosen orientation of $\beta$ for $i = 1, \cdots, n$.
        \item The tail of $\tilde{e}_n$ is the head of $\tilde{e}_1$.
        \item The tail of $\tilde{e}_i$ is the head of $\tilde{e}_{i+1}$ for $i = 1, \cdots, n-1.$
    \end{itemize}

    By Definition \ref{def:spike}, a half-edge $e_i$ corresponds to a spike if and only if its projection $\pi_E(e_i)$ is contained in $E^0(\vec{k})$. 
    For $i = 1, \cdots, n-1$, let $g_i$ (resp. $g_n$) be the common complete geodesic component of the spikes corresponding to $e_i$ and $e_{i+1}$ (resp. $e_n$ and $e_1$). 
    By Definition \ref{def:crown.bdy.}, the sequence $(g_1, \cdots, g_n)$ represents a crowned boundary of the hyperbolic surface.
\end{proof}

\begin{figure}[!h]
    \centering
    \includegraphics[width=0.7\linewidth]{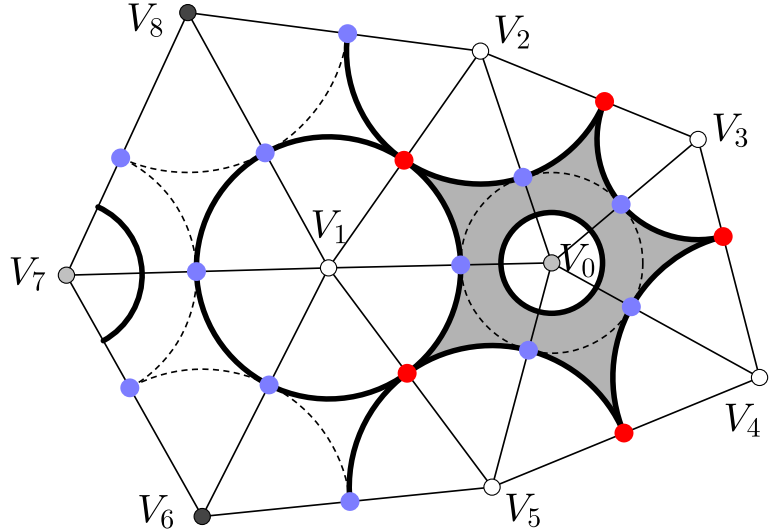}
    \caption{\small 
    A typical local picture of a degenerated circle-packed surface is shown. 
    Here, $k_i = 0$ for $i = 1, 2, 3, 4, 5$ and $k_i > 0$ for $i = 0, 6, 7, 8$. 
    All bold circles represent hyperbolic geodesic boundaries. 
    The shaded region, which acts as a connected component of the degenerated surface, is a complete hyperbolic surface with one crowned boundary and one closed geodesic boundary.
    }
    \label{fig:dgn_packing1}
\end{figure}

\begin{remark}
    A hyperbolic surface with a crowned boundary is called a crowned hyperbolic surface. 
    One motivation for studying crowned hyperbolic surfaces is that the metric completion of the complement of a geodesic lamination on a hyperbolic surface can result in a crowned hyperbolic surface. 
    Gupta applied crowned hyperbolic surfaces to construct limiting harmonic maps \cite{G19}. 
    Tamburelli provided a correspondence between wild globally hyperbolic maximal anti-de Sitter structures and pairs of crowned hyperbolic surfaces \cite{T21}.
\end{remark}

\section{Extended homeomorphism}\label{4}

In this section, we prove Theorem \ref{Main_result_1} which extends $\Phi : \Omega_k \to \Omega_T$ as a homeomorphism $\widehat{\Phi}: \overline{\Omega}_k \to \widehat{\Omega}_T$. We first obtain the homeomorphic extensions of the maps $\Phi^\Delta$ for generalized triangles (Proposition \ref{prop:local_homeo}). Next, we show that the extended map $\widehat{\Phi}$ is a homeomorphism on every stratum of the boundary (Proposition \ref{prop:strata}). Hence $\widehat{\Phi}$ is a continuous bijection. Finally we show that its inverse map is also continuous.

\subsection{Extended homeomorphisms for generalized triangles}\label{sec:4}

It was shown in \cite{BHS} that the map
\begin{equation}\label{map:Phi_T_Delta}
\begin{split}
    \Phi^\Delta \ : \ \Omega_k^\Delta \ &\to \  \Omega_T^\Delta \subset \mathbb{R}_{>0}^3 \\
    (k_1,k_2,k_3) & \mapsto (T_1,T_2,T_3)
\end{split}
\end{equation}
recording the total geodesic curvature of arcs in a triangle is a homeomorphism, where 
$$ \Omega_T^\Delta := \{\ (T_1,T_2,T_3)\in\mathbb{R}_{>0}^3 \ \lvert\  T_1+T_2+T_3<\pi\ \} \ .$$ 
In this subsection, we show that it extends continuously to a homeomorphism from $\overline{\Omega}_k^{\Delta} := \mathbb{R}_{\geq0}^3$ to $\widehat{\Omega}_T^{\Delta}$ (Proposition \ref{prop:local_homeo}), where
$ \widehat{\Omega}_T^{\Delta} := \{\ (T_1,T_2,T_3)\in\mathbb{R}_{\geq0}^3 \ \lvert\  T_1+T_2+T_3<\pi\ \} $ is a partial closure of $\Omega_T^\Delta$, similar to \eqref{def:widehat_T}.

The first ingredient is the following precise formula for the map $\Phi^\Delta$.
\begin{theorem}
    \label{thm:T123}
    For a circle packed triangle with curvature $(k_1,k_2,k_3)$, let $T_i$ be the total geodesic curvature of the circle arc centered at vertex $v_i,\ i=1,2,3$.
    Then for any cyclic re-ordering $(i,j,k)$ of $(1,2,3)$, when none of $k_1, k_2, k_3$ is $1$, we have 
        $$ T_i = T_i(k_1,k_2,k_3) = \frac{ \theta_i k_i}{\sqrt{\lvert 1-k_i^2 \rvert}}\ , $$
    where
    \begin{equation}\label{eq:k123}
        \frac{2+(k_i k_j + k_j k_k + k_k k_i) - k_i^2}{(k_i+k_j)(k_i+k_k)} = \left\{ 
        \begin{array}{ll}
             \cos \theta_i &  ,\ k_i>1  \\
             \cosh \theta_i & ,\ k_i<1 
        \end{array} \right. \ .
    \end{equation}
    
    Furthermore, the extension of the map given by 
    $$(k_1, k_2, k_3) \mapsto \big( T_1(k_1,k_2,k_3), T_2(k_1,k_2,k_3), T_3(k_1,k_2,k_3) \big)$$
    on $\Omega_k^\Delta=\mathbb{R}^3_{>0}$ is a $C^1$-map which is identical to the map $\Phi^{\Delta}$ of \eqref{map:Phi_T_Delta}. 
    
    In particular, for $k_i=1$, we have 
    $$ T_i = \frac2{\sqrt{(k_j+1)(k_k+1)}} \ . $$
    \end{theorem}
    
    \begin{corollary}\label{coro:k_P}
    Let $C_P$ be the circle passing through the points of tangency a circle packed triangle with curvature $(k_1,k_2,k_3)$. Then the geodesic curvature $k_P$ of $C_P$ is given by 
    \begin{equation}\label{eq:k_P}
    k_P^2 = k_1 k_2 + k_2 k_3 + k_3 k_1 + 1 .
    \end{equation}
    \end{corollary}  
    
    Note that the $C_P$ in Corollary \ref{coro:k_P} is orthogonal to  three circles. In \cite{BHS}, $C_P$ plays an important role in the relations for geometric quantities of a triangle.

    \begin{remark}
    We will proved Theorem \ref{thm:T123} and Corollary \ref{coro:k_P} in Appendix \ref{sec:app}, by applying the Klein model of hyperbolic plane.
    \end{remark}

\bigskip
With the help of the explicit formula in Theorem \ref{thm:T123}, we obtain the asymptotic behavior of total geodesic curvatures for triangles as follows. 
For any subset $I\subset\{1,2,3\}$, the strata $\partial^I\Omega_k^\Delta$ and $\partial_0^I\Omega_T^\Delta$ are defined similarly as (\ref{def:stratum}).   
    \begin{proposition}\label{prop:cntnu_extd}
        The homeomorphism $\Phi^\Delta: \Omega_k^\Delta \to \Omega_T^\Delta$ extends to a continuous map $\widehat\Phi^\Delta: \overline{\Omega}_k^{\Delta} \rightarrow \widehat{\Omega}_T^{\Delta}$,  
        with the image of $\partial^{I} \Omega_k^\Delta$ lying in  $\partial_0^I \Omega_T^\Delta$ for any $I\subset\{1,2,3\}$.        
     \end{proposition}
    \begin{proof}
        
        For any $(k_1^*,k_2^*,k_3^*)\in\partial^I\Omega_k^\Delta$ and any sequence $\{(k^n_1,k^n_2,k^n_3)\}_{n=1}^{\infty}\subset\Omega_k^\Delta$ converging to $(k_1^*,k_2^*,k_3^*)$, we claim that there exists a unique $T_i^* \in [0, \pi)$ with $T_i(k_1^n,k_2^n,k_3^n)\rightarrow T_i^*$ for $i = 1, 2, 3$. 
        The equation \eqref{eq:k123} implies that there exists $T_i^*\in(0,\pi)$ such that $T_i(k_1^n,k_2^n,k_3^n)\rightarrow T_i^*$ for any $i\notin I$. 
    With loss of generality, we assume that $0<k_i^n<1$ since $k_i^n\to 0^+$ for $i \in I$. 
        
        Then 
        \[ \cosh\theta_i^n = 1 + \frac{2(1-(k_i^n)^2)} {(k_i^n+k_j^n)(k_i^n+k_k^n)} \leq 1 + \frac{2(1-(k_i^n)^2)}{(k_i^n)^2} \leq \frac{2}{(k_i^n)^2}\ .  \]
        Combining the fact that   
        $\theta_i^n \leq 2 \ln \frac{2}{k_i^n}$, we have
        \[ \theta_i^n k_i^n \leq 2 k_i^n \ln \frac{2}{k_i^n} \to 0 \quad \textrm{as\ } n\to \infty , \]
        and 
        $$\lim_{n\rightarrow \infty} T_i(k_1^n,k_2^n,k_3^n) = 0 .$$
        We obtain the claim by taking $T_i^*$ as $0$ for $i \in I$. 
     
     By applying the Gauss-Bonnet formula on the dual triangle $\Delta^*$ of the limiting degenerated circle packed triangle, we have 
     \begin{equation}
     T_1(k_1^*,k_2^*,k_3^*)+T_2(k_1^*,k_2^*,k_3^*)+T_3(k_1^*,k_2^*,k_3^*)=\pi-\mathrm{Area}(\Delta^*)
     .
     \end{equation}     
    Since the dual triangle $\Delta^*$ never vanishes, $\sum_{ i\notin I} T_{i}(k_1^*,k_2^*,k_3^*)<\pi$. 
    Therefore, the image of $ \partial^{I} \Omega_k^\Delta$ is contained in $\partial_0^I \Omega_T^\Delta$.
\end{proof}

\bigskip
Next, we show that the map $\widehat\Phi^\Delta:\overline{\Omega}_k^{\Delta} \rightarrow \widehat{\Omega}_T^{\Delta}$ is a homeomorphism by variational principle. 
Similar to the method used in \cite{BHS}, we need the following result.

\begin{lemma}[\cite{DGL}, Lemma 5.4]\label{lem:embedding}
    Let $W:\mathbb{R}^n\to\mathbb{R}$ be a smooth strictly convex function with positively definite Hessian matrix, then $\triangledown W:\Omega \to \mathbb{R}^n$ is a smooth embedding.  
\end{lemma}

The desired convex function is obtained by the integral of a closed 1-form. 
For a proper (possibly empty) subset $I$ of $\{1,2,3\}$, let 
\begin{equation}\label{map:ln_loc}
\begin{split}
    \ln :\quad \partial^I \Omega_k^\Delta \  & \longrightarrow \quad \mathbb{R}^{3-\lvert I \rvert} \\
    (k_1,k_2,k_3) &\longmapsto \vec{S}=(S_j)_{j \notin I}
\end{split}
\end{equation}
be the diffeomorphism defined on the boundary stratum, 
where $S_j:=\ln k_j$ for $j \notin I$.
Define a 1-form on $\mathbb{R}^{3-\lvert I \rvert}$ as 
\begin{equation}\label{def:1form_loc}
    \omega_\Delta^I := \sum_{j\notin I} T_j(e^{S_1},e^{S_2},e^{S_3}) \mathrm{d}S_j,
\end{equation} 
where $e^{S_i}$ is appointed to be zero if $i\in I$.

\begin{lemma}\label{lem:closed_form}
    For any proper $I\subset \{1,2,3\}$, the 1-form $\omega_\Delta^I$ is closed on $\mathbb{R}^{3-\lvert I \rvert}$. 
\end{lemma}
\begin{proof}
    For $I=\varnothing$, it has been shown in \cite[Lemma 2.8]{BHS}. 
    
    If $I$ is consisted of two elements, then $\mathbb{R}^{3-|I|}$ is $\mathbb{R}$ and $\omega_\Delta^I$ is naturally closed on $\mathbb{R}$. 
    
    It remains to consider the case that $I$ has only one element.
    Without loss of generality, set $I=\{3\}$. 
    Let $l_i := T_i / k_i$ be the length of the arcs for $i = 1, 2, 3$. 
    We have 
    \[ \pp{T_1}{S_2} = \pp{T_1}{k_2} \frac{\mathrm{d}k_2}{\mathrm{d}S_2} = k_2 \pp{T_1}{k_2} = k_1k_2 \pp{l_1}{k_2} \quad 
    \textrm{and } \quad 
    \pp{T_2}{S_1} = k_1k_2 \pp{l_2}{k_1} \ .
    \]
    It sufficies to prove that $\pp{l_1}{k_2} = \pp{l_2}{k_1}$. 
    This can be done by a direct computation and we left it at the last of Appendix \ref{AppSec:2}. 
\end{proof}

Now for any proper $I\subset \{1,2,3\}$, pick a base point $\vec{S}_I \in \mathbb{R}^{3-|I|}.$
By Lemma \ref{lem:closed_form} and Theorem \ref{thm:T123}, the integral of the 1-form  
\begin{equation}\label{def:energt_loc}
    \mathcal{E}^I(\vec{S}) := \int_{\vec{S}_I} ^ {\vec{S}}  \omega_\Delta^I
\end{equation}
is a well-defined $C^2$-function on $\mathbb{R}^{3-\lvert I \rvert}$. We call it the \DEF{energy function} of index $I$.

By generalizing \cite[Lemma 2.11]{BHS} from empty set $I$ to proper set $I$, we state the following lemma. We omit the proof, since it is similar to  \cite[Lemma 2.11]{BHS} . 
\begin{lemma}\label{lem:convex} 
For any proper $I\subset \{1,2,3\}$,  $\mathcal{E}^I$ is strictly convex on $\mathbb{R}^{3-\lvert I \rvert}$.
\end{lemma}

Now all the conditions of Lemma \ref{lem:embedding} are satisfied. To give the exact image of the extended map $\widehat{\Phi}^{\Delta}$, we need the following asymptotic behaviour.

\begin{lemma}\label{lem:limit}
    Let $I\subset\{1,2,3\}$ be non-empty. 
    If $k_i \to +\infty$ for any $i\in I$ and $k_i$ is bounded from above for any $i \notin I$, then 
    $ \sum_{i \in I} T_i(k_1,k_2,k_3) \to \pi $.
    
\end{lemma}
\begin{proof}
    If $I=\{i\} $, by \eqref{eq:k123},
    $$ \cos\theta_i = \frac{2/k_i^2 + (k_j+k_k)/k_i + k_j k_k / k_i^2 -1}{(1+k_j/k_i)(1+k_k/k_i)} \to -1 \quad \textrm{as}\  k_i\to+\infty. $$
    Hence $\theta_i \to \pi$ and then $T_i \to \pi$.

    If $\lvert I \rvert \geq 2$, there are at least two of $k_i$'s tend to $+\infty$. 
    By Corollary \ref{coro:k_P}, we have
    $$ k_P^2 = k_1k_2 + k_2k_3 + k_3k_1 +1 \to +\infty \ . $$
    Hence the circle $C_P$ collapses to a point, and so does the dual triangle $\Delta^*$ which lies inside the dual triangle. By Gauss-Bonnet formula, 
    $T_1+T_2+T_3=\pi- \textrm{Area}(\Delta^*) \to \pi$.
    If $\lvert I \rvert =2$, we may assume $I=\{1,2\}$ and by \eqref{eq:k123} again we have
    \[ \cos\theta_3 \ \textrm{or} \ \cosh\theta_3 = \frac{2/k_1k_2+1+k_3/k_1+k_3/k_2-k_3^2/k_1k_2}{(1+k_3/k_1)(1+k_3/k_2)} \to 1 \quad \textrm{as}\  k_1,k_2\to+\infty . \]
    Hence $T_3\to 0$ in this case, and $T_1+T_2\to\pi$.
\end{proof}
Note that the limit of individual $T_i$ may not always exist. 

\begin{proposition}\label{prop:loc_homeo_strat}
    For any $I\subset \{1,2,3\}$, $\widehat{\Phi}^{\Delta}$ maps $\partial^{I} \Omega_k^\Delta$ homeomorphically onto $\partial_0^I \Omega_T^\Delta$.
\end{proposition}
\begin{proof}
    Obviously, $\widehat{\Phi}^{\Delta}$ maps $(0,0,0)$ to $(0,0,0)$. 
    Let $\vec{e}_1=(\pi,0,0), \vec{e}_2=(0,\pi,0), \vec{e}_3=(0,0,\pi)$ be the other three vertices of $\Omega_T^\Delta$.

    For any proper $I\subset\{1,2,3\}$, by the definition of energy function $\mathcal{E}^I$ in \eqref{def:1form_loc} and \eqref{def:energt_loc}, we have
    $$ \frac{\partial \mathcal{E}^I}{ \partial S_j} = T_j(e^{S_1},e^{S_2},e^{S_3}) \ ,\quad \textrm{for all } j\notin I \ . $$ 
    By Lemma \ref{lem:embedding} and Lemma \ref{lem:convex}, $\nabla \mathcal{E}^I$ is a smooth embedding of $\mathbb{R}^{3-\lvert I \rvert}$. Composing with the diffeomorphism $\ln$, $\widehat{\Phi}^{\Delta}$ is a homeomorphism of $\partial^{I} \Omega_k^\Delta$ onto its image. 

    Now for any $j\notin I$, let $\{(k_1^n, k_2^n, k_3^n)\}_{n=1}^\infty$ be a sequence in $\partial^{I} \Omega_k^\Delta$ satisfying the first case of Lemma \ref{lem:limit}. Together with Guass-Bonnet formula, we have 
    $$\lim_{n\to\infty} \widehat{\Phi}^{\Delta}(k_1^n, k_2^n, k_3^n) = \vec{e_j} \ .$$ 
    It is a basic fact that the gradient of a strictly convex $C^1$-function on $\mathbb{R}^n$ is a homeomorphism from $\mathbb{R}^n$ to a convex domain in $\mathbb{R}^n$. 
    Thus, $\widehat{\Phi}^{\Delta} \left( \partial^{I} \Omega_k^\Delta \right)$ must contain the open convex hull of $(0,0,0)$ and $\vec{e}_j$ for $j\notin I$. But this is exactly $\partial_0^I \Omega_T^\Delta$. Hence, with Proposition \ref{prop:cntnu_extd}, $\widehat{\Phi}^{\Delta} \left( \partial^{I} \Omega_k^\Delta \right) = \partial_0^I \Omega_T^\Delta$ and the map is a homeomorphism.
\end{proof}

\begin{proposition}\label{prop:local_homeo}
    The extended map $\widehat{\Phi}^{\Delta} : \overline{\Omega}_k^{\Delta} \to \widehat{\Omega}_T^{\Delta}$ is a homeomorphism.
\end{proposition}
\begin{proof}
    By Proposition \ref{prop:cntnu_extd} and \ref{prop:loc_homeo_strat}, $\widehat{\Phi}^{\Delta}$ is a continuous bijection from $\overline{\Omega}_k^{\Delta} $ to $\widehat{\Omega}_T^{\Delta}$.
    We only need to show that the inverse map $(\widehat{\Phi}^{\Delta})^{-1}$ is also continuous. 
    Let $\{\vec{T}^n\}_{n=1}^{+\infty}$ be a sequence in $\widehat{\Omega}_T^{\Delta}$ that converges to some $\vec{T}^* \in \partial_0^I \Omega_T^\Delta$. Let $\vec{k}^n$ and $\vec{k}^*$ be the preimage of $\vec{T}^n$ and $\vec{T}^*$ under $\widehat{\Phi}^{\Delta}$, respectively. Then we only need to show that $\lim_{n\to\infty} \vec{k}^n = \vec{k}^*$. 
    
    We first show that $\{\vec{k}^n\}_{n=1}^\infty$ is bounded in $\overline{\Omega}_k^{\Delta} = \mathbb{R}_{\geq0}^3$. If not, by passing to a subsequence, we may assume that $\{\vec{k}^n\}$ satisfy the condition of Lemma \ref{lem:limit} for some non-empty $J\subset \{1,2,3\}$. 
    Then $T_1^*+T_2^*+T_3^* = \lim(T_1^n+T_2^n+T_3^n) = \pi$ by Lemma \ref{lem:limit}. This contradicts the assumption $\lim \widehat{\Phi}^\Delta (\vec{k}_n) = \vec{T}^* \in \partial_0^I \Omega_T^\Delta$ that requires $T_1^*+T_2^*+T_3^* <\pi$.
    
    If $\lim \vec{k}^n \neq \vec{k}^*$, since $\{\vec{k}^n\}_{n=1}^{+\infty}$ is bounded, there exists a subsequence $\{\vec{k}^{n_{i}}\}_{i=1}^{+\infty}$ such that 
    $$ \lim_{n_i\to\infty} \vec{k}^{n_{i}} = \vec{k'} \neq \vec{k}^* \ . $$ 
    Since $\widehat{\Phi}^{\Delta}$ is continuous, we have 
    $$ \widehat{\Phi}^{\Delta} (\vec{k'})= \lim_{n_i\to\infty} \widehat{\Phi}^{\Delta}(\vec{k}^{n_i}) = \lim_{n_i\to\infty}\vec{T}^{n_i} = \vec{T}^* \ . $$ 
    However, $\widehat{\Phi}^{\Delta}(\vec{k}^*) = \vec{T}^*$ by assumption. This contradicts the bijectivity of $\widehat{\Phi}^{\Delta}$. Hence $(\widehat{\Phi}^{\Delta})^{-1}$ is continuous. This finishes the proof completely.
\end{proof}

\subsection{Homeomorphism in every stratum}
Now we turn to circle packed surfaces. 
The results are essentially obtained by summing up local results for three-circle configurations.

By Proposition \ref{prop:cntnu_extd} and the Gauss-Bonnet formula, the homeomorphism $\Phi: \Omega_k \to \Omega_T$ extends continuously to $\widehat{\Phi}: \overline{\Omega}_k \to \widehat{\Omega}_T$, with $\widehat{\Phi}(\partial^I \Omega_k) \subset \partial_0^I \Omega_T$ for any non-empty $I \subset V$. In this subsection, we obtain the following proposition, which is a global version of Proposition \ref{prop:loc_homeo_strat}. The proof is inspired by the work of 
Bobenko-Springborn \cite{bo}. 

\begin{proposition}\label{prop:strata}
    For any $I\subseteq V$, the restricted map $\widehat{\Phi}: \partial^I\Omega_k \to \partial^I_0 \Omega_T$ is a homeomorphism.
\end{proposition}

\begin{remark}
    If $I=V$, $\widehat{\Phi}$ maps the origin point to the origin point. 
    By taking $I = \varnothing$, Proposition \ref{prop:strata} implies Theorem \ref{mr}. 
    This subsection also provides an alternative proof.  
\end{remark}

Proposition \ref{prop:strata} will be proved as following: 
  
Denote $v\in f$ if $v\in V$ is a vertex of triangle $f\in F$.
For a fixed proper set $I\subset V$, let $\overline{V}:=V\setminus I$ and $\overline{F}:=F_{V\setminus I}$. Note that $f\in \overline{F}$ if and only if $f$ contains at least one vertex outside $I$. Since $I\neq V$, both $\overline{V}$ and $\overline{F}$ are non-empty.
Let
\begin{equation}\label{def:ln}
    \ln : \partial^I\Omega_k \to \Omega^I_S := \mathbb{R}^{\overline{V}}
\end{equation}
be the diffeomorphism of the stratum defined by $\vec{k}=(k_i)_{i\in V} \mapsto \vec{S}=( \ln k_i )_{i \notin I}$. 
We will define a set $\mathcal{CO}^I$, called ``coherent system'' (Definition \ref{def:CO}), with a natural embedding $\iota : \Omega_S^I \hookrightarrow \mathcal{CO}^I$ \eqref{map:iota} and a projection $\Psi: \mathcal{CO}^I \to \partial_0^I\Omega_T$ \eqref{map:psi}, such that on each stratum $\partial^I \Omega_k$, the map $\widehat{\Phi}$ is expressed as a composition
\begin{equation}\label{diagram_spaces}
\begin{tikzcd}[row sep=large, column sep=large]  
\partial^I\Omega_k \arrow[r,"\ln","\cong"'] \arrow[rrr,to path={|- ([yshift=-3.25ex]\tikztostart.south) -- node[above, pos=3.2]{\small $\widehat{\Phi}= \Psi \circ \iota \circ \ln$}([xshift=3ex,yshift=-5ex]\tikztostart.east) -| (\tikztotarget.south)}] & \Omega_S^{I} \arrow[r, hookrightarrow, "\iota"] & \mathcal{C}O^I \arrow[r, "\Psi"] & \partial_0^{I}\Omega_T  \ . 
\end{tikzcd}  
\end{equation}

Proposition \ref{prop:strata} is then deduced from the following two propositions, which will be proved in \S \ref{sssec:coh.sys.} and \S \ref{sssec:feas.flow} respectively.
\begin{proposition}\label{prop:image}
    $\widehat{\Phi}$ maps $\partial^I\Omega_k$ homeomorphically onto $\Psi(\mathcal{CO}^I)$.
\end{proposition}

\begin{proposition}\label{prop:co}
    $\Psi(\mathcal{CO}^I) = \partial_0^I\Omega_T$.
\end{proposition}

\subsubsection{Coherent system and image set}
    \label{sssec:coh.sys.}
\newcommand{\CO}{\mathcal{CO}}
\begin{definition}\label{def:CO}
    The \DEF{coherent system} is a subset of $\mathbb{R}_{\geq 0}^{3\lvert F \rvert}$ defined as \begin{equation}
        \CO:= \bigg\{\ \left( T_v^f \right)_{v \in f} \ \bigg\lvert\ 
        T_u^f + T_v^f +T_w^f < \pi\ \textrm{ for any } f=(uvw)\in F \ \bigg\} \ . 
    \end{equation}
    And for any $I\subseteq V$, let \begin{equation}
        \CO^I := \bigg\{\ \left( T_v^f \right)_{v \in f} \in \CO \ \bigg\lvert\ 
        T_v^f =0 \textrm{ if and only if } v\in I \ \bigg\} \ .
    \end{equation}
\end{definition}
Geometrically, a point in $\CO$ realizes each $f$ as a possibly degenerated triangle, so that the total geodesic curvature of circle arc centered at $v\in f$ is exactly $T_v^f$ if $T_v^f>0$, and is geodesic if $T_v^f=0$. 
These triangles do not need to match along the edges.

We follow notations in \eqref{def:T.g.c_v^f}, \eqref{def:T.g.c_v} and \eqref{def:ln}. There is a natural embedding 
\begin{equation}\label{map:iota}
\begin{split}
    \iota :\quad \Omega_S^I \quad &\longrightarrow \quad \CO^I \\
    \vec{S}=\ln \vec{k} &\longmapsto \left( T_v^f(\vec{k}) \right) _{v \in f} 
\end{split}
\end{equation}
where $\vec{k}\in \partial^I\Omega_k$, and there exists a natural projection
\begin{equation}\label{map:psi}
\begin{split}
    \Psi :\quad \CO \quad &\longrightarrow \quad \widehat{\Omega}_T \\
    \left( T^f_v \right)_{v \in f} &\longmapsto \left( \sum_{f \in F_{\{v\}} } T_v^f \right) _{v \in V} 
\end{split}
\end{equation}
by summing up at vertices. Note that $\Psi(\CO^I) \subseteq \partial_0^I\Omega_T$ by the Gauss-Bonnet formula. 

        

In order to prove Proposition \ref{prop:image} which stated that $\widehat{\Phi}$ is a surjective homeomorphism from $\partial^I\Omega_k$ to $\Psi(\mathcal{CO}^I)$, we need the following result. 
We say $f:\mathbb{R}^n \to \mathbb{R}$ is \DEF{coercive} if $f(x) \to +\infty$ as $\| x \| \to +\infty$. 
\begin{lemma}[\cite{Gexu}, Lemma 4.6]\label{lem:coercive}
    Suppose $f(x)$ is a $C^2$-smooth, strictly convex function on $\mathbb{R}^n$.
    The following two statements are equivalent:
    \begin{enumerate}[label=(\arabic*)]
      \item $f$ has a unique minimum;
      \item $f$ is coercive.
    \end{enumerate}
\end{lemma}

\newcommand{\engy}{\widetilde{\mathcal{E}}^I}
\newcommand{\engyf}{\widetilde{\mathcal{E}}^I _f}
\begin{proof}[Proof of Proposition \ref{prop:image}]
    It is easy to see that $\widehat{\Phi} = \Psi \circ \iota \circ \ln : \partial^I\Omega_k \to \partial_0^I\Omega_T$. Hence $\widehat{\Phi}(\partial^I\Omega_k) \subset \Psi(\CO^I)$.

    Now we fix a point $\vec{T}=(T_v)_{v \in V}$ in $\Psi(\CO^I)$, with a preimage $(T_v^f)_{v \in f}$ in $\CO^I$. 
    For a triangle $f=(uvw)\in \overline{F}$, we can define a closed 1-form $\omega_f^I$ on $\mathbb{R}^{\{u,v,w\}\setminus I}$ as (\ref{def:1form_loc}). It naturally induces a closed 1-form on $\Omega_S^I = \mathbb{R}^{V \setminus I}= \mathbb{R}^{\overline{V}}$.
    However, we will modify this 1-form to obtain global results.
    Define another 1-form 
    \begin{equation}
        \widetilde{\omega}_f^I := \omega_f^I - \left( T_u^f \mathrm{d}S_u + T_v^f \mathrm{d}S_v + T_w^f \mathrm{d}S_w \right) 
    \end{equation}
    on $\mathbb{R}^{\{u,v,w\}\setminus I}$, where $T_i^f\mathrm{d}S_i$ is defined to be zero when $i \in I$. Then let 
    \begin{equation}
        \engyf := \int _ {\vec{S}_0^I} ^ {\vec{S}^{I}} \widetilde{\omega}_f^I.
    \end{equation}
     Finally, define the energy function as 
    \begin{equation}\label{engyy}
        \engy := \sum _{f \in \overline{F}} \engyf.
    \end{equation}
    By \eqref{def:1form_loc} and the definitions above, when $\ln \vec{k} = \vec{S} \in \Omega_S^I$ and $v\in \overline{V}$, 
    \begin{equation}\label{eq:deriv_engy}
    \frac{\partial \engy}{\partial S_v} (\vec{S}) 
    = \sum_{f \in \overline{F}} \frac{\partial \engyf}{\partial S_v} (\vec{S}) 
    = \sum_{f \in F_{\{v\}}} \left( T_v^{f}(\vec{k}) - T_v^f \right) 
    = T_{v}(\vec{k}) - \left( \sum_{f \in F_{\{v\}}} T_v^f \right) = T_v(\vec{k}) - T_v \ .
    \end{equation}
    In other words, $\nabla \engy (\ln \vec{k}) = \widehat{\Phi} (\vec{k}) - \vec{T}$ as a vector-valued function on $\partial^I\Omega_k$. 

    \bigskip
    \textbf{Claim:}
    $\engy$ is a strictly convex and coercive function on $\Omega_S^I = \mathbb{R}^{\overline{V}}$. 
    
    Hence $\engy$ contains a unique minimum point $\vec{S}_{0}=\ln(\vec{k}_{0})$ in $\Omega_S^I$ by Lemma \ref{lem:coercive}, and $\frac{\partial \engy}{\partial S_v} (\vec{S}_{0}) = 0 $ for any $v \in \overline{V} $. 
    On the other hand, by \eqref{eq:deriv_engy} 
    we have $\frac{\partial \engy}{\partial S_v} (\vec{S}_0) = T_v(\vec{k}_0) - T_v$.     
    This implies $\widehat{\Phi}(\vec{k}_{0}) = \vec{T} = \Psi( (T_v^f)_{v \in f} ) $. Therefore, we have $\Psi(\CO^I) = \widehat{\Phi}(\partial^I\Omega_k)$.

    Since $\engy$ is strictly convex from $\mathbb{R}^{\overline{V}}$, by Lemma \ref{lem:embedding}, $\nabla\engy$ is a smooth embedding. By (\ref{eq:deriv_engy}), $\nabla\engy$ is simply differed from the restriction of $\widehat{\Phi}$ by a constant translation. Therefore, $\widehat{\Phi}$ also maps the stratum homeomorphically onto its image.
    \end{proof}

    \begin{proof}[Proof of Claim]
    Let $f=(uvw)\in \overline{F}$ as before. By Lemma \ref{lem:convex}, as a $C^2$-smooth function on $\mathbb{R}^{\{u,v,w\}\setminus I}$, $\engyf$ is the sum of a strictly convex function and a linear function, hence strictly convex again. 
    Since $(T_v^f)_{v \in f} \in \CO^I$, there exists $\vec{k}_f \in \overline{\Omega} _k^\Delta$ such that $\widehat{\Phi}_T^{\Delta} (\vec{k}_f) = (T^f_u, T^f_v, T^f_w)$, by Proposition \ref{prop:local_homeo}. 
    Let $\vec{S}_f:= \ln(\vec{k}_f) \in \mathbb{R}^{\{u,v,w\}\setminus I}$, then $\vec{S}_f$ is a zero of the 1-form $\widetilde{\omega}_f^I$. By Lemma \ref{lem:coercive}, $\engyf$ is coercive when regarded as a function on $\mathbb{R}^{\{u,v,w\}\setminus I}$. 
    It also has a unique minimum.
    
    However, as a $C^2$-function on $\Omega_S^I$, $\engyf$ is merely convex, and coercive along certain directions. More precisely, let $H^I_f$ be the $\lvert\overline{V}\rvert \times \lvert\overline{V}\rvert$ Hessian of $\engyf$, and $\vec{S}=(S_i)_{i \in \bar{V}}\in\mathbb{R}^{\overline{V}}$ be a column vector. 
    We have the following four statements: 
    \begin{enumerate}[label=(\arabic*)]
        \item for any non-zero vector $\vec{S}\in\mathbb{R}^{\overline{V}}$, we have ${\vec{S}}^T H^I_f \vec{S} \geq 0$; 
        \item if $S_i\neq0$ for some $i\in\{u,v,w\}\setminus I$, then ${\vec{S}}^T H^I_f \vec{S} >0$; 
        \item $\engyf \geq m_f$ for some constant $m_f$;
        \item if $\lvert S_i \rvert \to +\infty$ for some $i\in\{u,v,w\} \setminus I$, then $\engyf(S)\to +\infty$.
    \end{enumerate}

    Now let $H^I$ be the $\lvert\overline{V}\rvert \times \lvert\overline{V}\rvert$ Hessian of $\engy$. For any non-zero vector $\vec{S}\in\mathbb{R}^{\overline{V}}$, choose $j \in \overline{V}$ with $S_j\neq0$. There exists at least one $f_0 \in F_{\{j\}} \subset \overline{F}$. By statements (i) and (ii) listed above, 
    $$\vec{S}^T H^I \vec{S} = \sum_{f\in \overline{F}} \vec{S}^T H^I_f \vec{S} \geq \vec{S}^T H^I_{f_0} \vec{S} > 0 \ . $$
    This implies $H^I$ is again positively definite and $\engy$ is strictly convex. 

    To prove that  $\engy$ is coercive on $\mathbb{R}^{\overline{V}}$, it sufficies to show that, for any $A > 0$, there exists $B > 0$ such that $\tilde{\mathcal{E}}^I (\vec{S}) > A$ holds for every $\vec{S} \in \mathbb{R}^{\bar{V}}$ with $\| \vec{S} \|_{\infty} > B$. 
    By statement (iv) listed above and the finiteness of $|F|$, for any large $A > 0$,  there is some constant $B>0$ such that 
    \begin{center} for all $f=(uvw)\in\overline{F}$, whenever $\lvert S_i \rvert > B$ for some $i\in\{u,v,w\} \setminus I$, we have \end{center} 
    $$\engyf(\vec{S})> A - \sum_{g \neq f \in \overline{F}} m_g \ . $$
    Therefore, for $\vec{S} \in\mathbb{R}^{\overline{V}}$ with $\| \vec{S} \|_{\infty} >B$, there exists a vertex $j\in\overline{V}$ such that $\lvert S_j \rvert > B$ and $f_0\in F_{\{j\}} \subset \overline{F}$. By the choice of $B$ and the lower bounds for $\engyf$'s, we have
    $$ \engy = \widetilde{\mathcal{E}}^I _{f_0} + \sum _{f_0 \neq f \in \overline{F}} \engyf > \left( A - \sum_{f_0 \neq g \in \overline{F}} m_g \right) + \sum _{f_0 \neq f \in \overline{F}} m_f = A \ . $$
\end{proof}

\subsubsection{Feasible flow and surjectivity}
    \label{sssec:feas.flow}
\newcommand{\Btimes}{\mbox{\scalebox{1.3}{$\boxtimes$}}}
To obtain Proposition \ref{prop:co} which implies that $\Psi:\CO^I \to \partial_0^I\Omega_T$ is actually surjective, we have to solve linear equations with inequality constrains. 
This turns out to be a problem of constructing a balanced flow on an oriented graph, which is a linear programming. The argument here is similar to \cite[\S 4.3]{bo}.

For a fixed proper set $I\subset V$, let $A(I)$ be the set of inner corners at a vertex outside $I$, or equivalently, 
\[ A(I) := \bigg\{\ (v,f)\in (V \setminus I) \times F \ \bigg\lvert\ v \in f \ \bigg\} \subset \overline{V}\times\overline{F} \ . \]
Let $\vec{M}=(M_v)_{v\in V}$ be a fixed point in $\partial_0^I\Omega_T$. 
An oriented graph $(N,X,\mathcal{I})$ associated to the triangulation $G$ and $I$ is defined as following. 
\medskip
    
    \textbullet\ The vertices set $N:= \overline{V} \sqcup \overline{F} \sqcup \{ \Btimes \}$, where $\Btimes$ is an extra virtual vertex.
    
    \textbullet\ There are three types of oriented edges in $X$.
    
    \quad (1). For any $v\in \overline{V}$, there is an oriented edge $x_v$ from $\Btimes$ to $v$. 
    
    \quad (2). For any $(v,f)\in A(I)$, there is an oriented edge $x_{v,f}$ from $v$ to $f$. 
    
    \quad (3). For any $f\in \overline{F}$,  there is an oriented edge $x_{f}$ from $f$ to $\Btimes$. 

\medskip
Next, assign each oriented edge $x\in X$ a closed interval $\mathcal{I}_x=[a_x, b_x]$ of $[- \infty, + \infty]$
\medskip
    
    \textbullet\ For any $v\in \overline{V}$, $\mathcal{I}_{x_v} := [M_v, M_v]$. 
    
    \textbullet\ For any $(v,f)\in A(I)$, $ \mathcal{I}_{x_{v,f}} := [\delta , +\infty)$. 
    
    \textbullet\ For any $f\in \overline{F}$,  $ \mathcal{I}_{x_f} := (-\infty, \pi - \varepsilon]$. 

Here $\delta, \varepsilon$ are two sufficiently small positive constants (depending on $\vec{M} \in\partial_0^I\Omega_T$) to be determined. This graph faithfully represents the linear equations and constrains.

\begin{definition}
    A \DEF{feasible flow} of $(N,X,\mathcal{I})$ is a function $\phi:X\to\mathbb{R}$ 
    such that 
    \begin{enumerate}
        \item it is balanced at each vertex $n \in N$, namely 
\[ \sum_{x\in in(n)} \phi(x) = \sum_{y\in ex(n)} \phi(y),\quad \forall n\in N \ ; \] 
        \item it is constrained by the intervals $\mathcal{I}_x$ on each edge $x\in X$, namely 
\[ a_x \leq \phi(x) \leq b_x,\quad \forall x \in X \ . \]
    \end{enumerate}
    For $n\in N$, $in(n)$ is the set of oriented edges ending at $n$, and $ex(n)$ is the set of oriented edges starting from $n$. 
\end{definition}
See Figure \ref{fig:feasible} as an example of the graph $(N,X,\mathcal{I})$. It is easy to check that a feasible flow is equivalent to a solution of $\Psi(\vec{T}) = \vec{M} \in \partial_0^I\Omega_T$, with $\vec{T} \in\CO^I$.

\begin{lemma}\label{lem:equiv_ff}
    For a given proper subset $I\subset V$ and $\vec{M}=(M_v)_{v\in V} \in \partial_0^I\Omega_T $, there exists some $\vec{T} = (T_v^f)_{v\in f} \in \CO^I$ with $\Psi(\vec{T})=\vec{M}$ if and only if there exists a feasible flow $\phi$ of $(N,X,\mathcal{I})$ associated to $G$ and $I$ for suitable $\varepsilon,\delta>0$ defined as above.  
\end{lemma}

\begin{figure}
    \centering
    \includegraphics[width=0.9\linewidth]{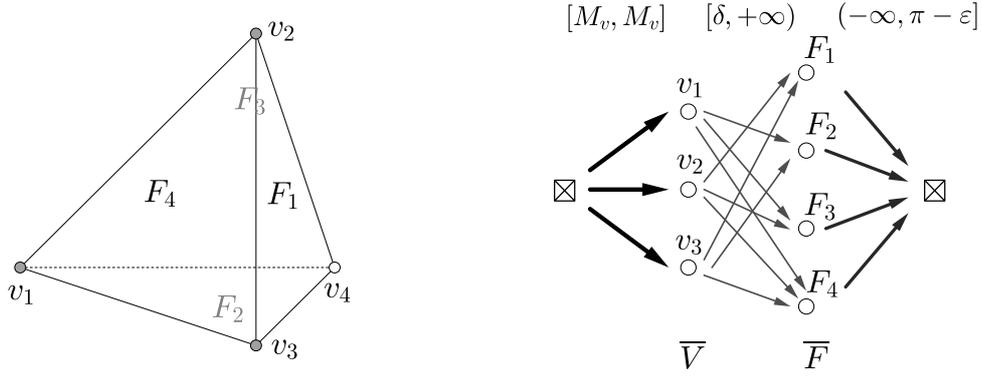}
    \caption{\small Left: $G$ is the tetrahedron triangulation of the sphere. $F_i$ is the face opposite to vertex $v_i$. Right: the oriented graph $(N,X,\mathcal{I})$ associated with $G$ and $I:=\{v_4\}$. The top row shows the closed intervals assigned to the edges. The two $\Btimes$'s are identified as the same vertex.}
    \label{fig:feasible}
\end{figure}

The following Feasible Flow Theorem gives a condition under which feasible flows exist. See \cite[\S 4.3]{bo}, \cite[Ch. II, \S 3]{FF}
    \begin{theorem}[Feasible Flow Theorem]\label{thm:FFT}
    There is a feasible flow of $(N,X,\mathcal{I})$ if and only if 
    \begin{equation}\label{eq:FFT}
        \sum_{x \in ex(N')} b_x \geq \sum_{y \in in(N')} a_y
    \end{equation}
    for any non-empty proper subset $N'\subset N$, where $ex(N')$ (resp. $in(N')$) is the set of oriented edges that start from (resp. end at) a vertex in $N'$ but end at (resp. start from) a vertex outside $N'$. 
\end{theorem}

\begin{proposition}\label{prop:always_ff}
    For any given proper subset $I\subset V$ and $\vec{M}=(M_v)_{v\in V} \in \partial_0^I\Omega_T $, there are suitable $\delta, \varepsilon >0$ such that the graph $(N,X,\mathcal{I})$ associated with $G$ and $I$ has a feasible flow. 
\end{proposition}
\begin{proof}
    It sufficies to check the Feasible Flow Theorem for every $N'\subset N$.

\vskip 0.3cm

    \noindent \uline{Case (A).} $\Btimes\in N'$.
    
    \indent \uline{Case (A.1).} If there exists $f\in \overline{F}\setminus N'$, then $x_f \in in (N')$ and $a_{x_f} = -\infty$. \eqref{eq:FFT} holds. 
    
    \indent \uline{Case (A.2).} If $\overline{F} \subset N'$, then the properness of $N'$ implies that $\overline{V} \setminus N'$ is not empty. 
    Then $ex(N')$ consists of $x_v$ with $v\in \overline{V}\setminus N'$, and $in(N')$ consists of $x_{v,f}$ with $v\in \overline{V}\setminus N'$ and $(v,f)\in A(I)$. Hence \eqref{eq:FFT} requires that
    \[  
    \sum_{v\in \overline{V}\setminus N'} M_v \geq \sum_{(v,f)\in A(I), v\in \overline{V}\setminus N'} \delta \ . 
    \]
    This will be satisfied if we let 
    \begin{equation}\label{eq:FFT_CA2}
        \delta < \frac{ \min_{v\in \overline{V} } \{M_v\} }{ \max_{v\in \overline{V}} \lvert F_{v} \rvert } \ .
    \end{equation}
    
\vskip 0.3cm

    \noindent \uline{Case (B).} $\Btimes\notin N'$. 
    Let $V':=\overline{V} \cap N'$ and $F':=\overline{F} \cap N'$.

    \indent \uline{Case (B.1).} If $V'=\varnothing$, then $N'=F'\neq \varnothing$. 
    We have $ex(N')=\{x_f \lvert f\in F'\}$, $in(N')=\{x_{v,f} \lvert f\in F', (v,f)\in A(I) \}$.
    Now (\ref{eq:FFT}) requires $(\pi-\varepsilon) \lvert F' \rvert \geq \delta \lvert in(N') \rvert$.
    Note that $\lvert in(N') \rvert \leq 3 \lvert F' \rvert$. Hence we only need to require $\pi>\varepsilon + 3\delta$, or 
    \begin{equation}\label{eq:FFT_CB1}
        0 < \delta, \varepsilon < \pi /4 \ .
    \end{equation} 

    \indent \uline{Case (B.2).} If there exists some $v\in V'$. If there further exists some $f \in F_{v}\setminus N'$, then $x_{v,f}\in ex(N')$ and $b_{x_{v,f}} = +\infty$, satisfying (\ref{eq:FFT}). So we may assume $F_{V'}\subset N' $. 
    On the other hand, if there is some $f\in F'$ with $v\notin V'$ for all $(v,f)\in A(I)$, then by the choice of $\delta, \varepsilon$ in Case (B.1), \eqref{eq:FFT} is locally satisfied at $f$.
    Hence we only need to consider the case where $F' = F_{V'}$.
    Now that $ex(N')=\{x_f\lvert f\in F_{V'}\}$. $in(N')$ contains two types of edges: $x_{v'}$ with $v'\in V'$, and $x_{v,f}$ with $v\in \overline{V}\setminus N', f\in F_{V'}, (v,f)\in A(I)$. 
    Note that $f\in F_{V'}$ contains at least one vertex in $V'$. 
    Therefore $\#\{x_{v,f}\lvert v\in \overline{V}\setminus N', f\in F_{V'} \} \leq 2 \lvert F_{V'} \rvert$.
    Now that 
    \begin{align*}
        \sum_{x \in ex(N')} b_x &= (\pi-\varepsilon)\lvert F_{V'} \rvert \ , \\ 
        \sum_{y \in in(N')} a_y &= \sum_{v'\in V'} M_{v'} + \delta\cdot \#\left\{x_{v,f}\left\lvert v\in \overline{V}\setminus N', f\in F_{V'} \right. \right\} 
        \leq \sum_{v'\in V'} M_{v'} + 2\delta \lvert F_{V'} \rvert \ .
    \end{align*}
    Hence we only need to require that $(\pi-\varepsilon)\lvert F_{V'} \rvert > \sum_{v'\in V'} M_{v'} + 2\delta \lvert F_{V'} \rvert $ for any non-empty $V'\subset \overline{V}$, which is equivalent to, by (\ref{def:eta}), 
    \[ \eta^{V'}(M) > (\varepsilon + 2\delta) \lvert F_{V'} \rvert \ . \]
    Since $M\in \partial_0^I\Omega_T \subset \widehat\Omega_T$, $\eta^{V'}(M) >0$ for all non-empty $V'\subset V$ by (\ref{def:widehat_T}). Hence 
    \begin{equation}\label{eq:FFT_CB2}
        \min_{V' \subsetneq \overline{V} } \left\{ \frac{\eta^{V'}(M)}{\left\lvert F_{V'} \right\rvert} \right\}
    \end{equation}
    is positive, and the choice for $\varepsilon+2\delta$ always exists.  
    
    Combining \eqref{eq:FFT_CA2}, \eqref{eq:FFT_CB1}, \eqref{eq:FFT_CB2}, the choice of positive real numbers $\delta, \varepsilon$ always exists.
\end{proof}

\begin{proof}[Proof of Proposition \ref{prop:co}]
    The existence of $\Psi$-preimages of any vector of $\partial_0^I \Omega_T$ follows from Lemma \ref{lem:equiv_ff} and Proposition \ref{prop:always_ff}. 
\end{proof}

\subsection{Global homeomorphism}
We give the proof of Theorem \ref{Main_result_1} as follows.
\begin{proof}[Proof of Theorem \ref{Main_result_1}]
    According to Proposition \ref{prop:strata}, we have shown that 
$$ \widehat{\Phi} : \partial^I \mathbb{R}_{>0}^V \to \partial_0^I \Omega_T $$
is a homeomorphism for every $I\subset V$. 
Then $\widehat{\Phi}$ is bijective. Since  $\widehat{\Phi}$ is continuous, to show that it is an actual homeomorphism, we need to show that the inverse map is also continuous. Denote by $\{\vec{T}^n\}_{n=1}^{+\infty}$ a sequence in $\widehat{\Omega}_T$ converging to some $\vec{T}^* \in \partial_0^I \Omega_T$. Set $\vec{k}^n$ and $\vec{k}^*$ the preimage of $\vec{T}^n$ and $\vec{T}^*$ under $\widehat{\Phi}$, respectively. Similar to the proof of Proposition \ref{prop:local_homeo}, we only need to show that  $\{\vec{k}^n\}_{n=1}^\infty$ is bounded in $\overline{\Omega}_k$. 
If not, by passing to a subsequence, 
we can assume that there exists a vertex subset $J\subset V$ such that $\lim k_j^n = +\infty$ for all $j\in J$, and $k_i^n$ is bounded from above for all $i \notin J$. 
The set $F_J$ can be decomposed as $F_{J} = F_{J_{1}} \sqcup F_{J_{2}} \sqcup F_{J_{3}}$, where $F_{J_{m}}$ is the set of faces having exactly $m$ vertex in $J$ ($m=1,2,3$). 
By Lemma \ref{lem:limit}, for $f\in F_{J_m}$, as $n\to \infty$, 
$$ \sum_{v\in f, v\in J} T_v^f(\vec{k}^n) \to \pi $$
and the sum contains exactly $m$ terms. So as $n\to\infty$, 
\[ \sum_{j\in J} T_j(\vec{k}^n) = \sum_{f\in F_{J}} \sum_{v\in f, v\in J} T_v^f(\vec{k}^n) \to \pi\lvert F_J \rvert \ . \]
This contradicts the assumption $\lim \vec{T}^n = \vec{T}^* \in \partial_0^I \Omega_T $ that requires $\eta^J(\vec{T}^*) = \pi\lvert F_J \rvert -\sum_{j\in J} T_j(\vec{k}^*) >0 $ in \eqref{Omega_T}. 
Hence $\widehat{\Phi}^{-1}$ is continuous and Theorem \ref{Main_result_1} is proved.
\end{proof}

\section{Combinatorial Ricci flow with prescribed total geodesic curvatures}\label{5}

In this section, we prove Theorem \ref{Main_result_3}, which states the convergence of the combinatorial Ricci flow \eqref{CRF1} with prescribed total geodesic curvatures $\vec{T} \in \partial_0^I \Omega_T$.

Let $J \subset I \subset V$ as before. 
Recall that for any $\vec{k} \in \partial^J \Omega_k$, $\vec{S} := \ln \vec{k} \in \Omega_S^I = \mathbb{R}^{V \setminus J}$ is defined as $(\ln k_i)_{i \notin J}$. 
By this change of variables, the flow \eqref{CRF1} transforms to
\begin{align}\label{CF2}
    \frac{d}{dt} S_v(t) = -\left( T_v(\vec{S}(t)) - T_v \right), \quad \text{for any } v \in V \setminus J,
\end{align}
which is the gradient flow of the potential function $\engy$ (\ref{engyy}) on $\Omega^I_S$. 
Here, $\vec{k}(t)$ is regarded as a vector-valued function of $\vec{S}$.

\begin{lemma}\label{solu}
    For any $\vec{k}(0) \in \partial^J \Omega_k$, the solution of the flow \eqref{CF2} with the initial value $\vec{S}(0) := \ln \vec{k}(0)$ exists for all time $t \in [0, \infty)$ and is unique.
\end{lemma}

\begin{proof}
Since $T_v$ is a $C^1$ function of $\vec{S}$, by Cauchy-Lipschitz Theorem, the flow (\ref{CF2}) exists a unique solution $\vec{S}(t)$ for some $t\in[0,\varepsilon)$. 
Since $\widehat{\Omega}_T$ is a bounded region in $\mathbb{R}^{V}$, 
\[ \left\vert\frac{dS_{v}}{dt}\right\vert
=\left\vert T_v(\vec{k}(t))-T_v\right\vert \]
is always bounded. 
Hence, the flow \eqref{CF2} exists for all $t\in [0,\infty)$. 
The uniqueness of the flow can be easily deduced from the classical theory of ordinary differential equations.
\end{proof}

In the spirit of the work of Takatsu \cite{Takatsu}, we introduce the important lemma for estimating the gradient of a convex function along its gradient flow, see e.g. \cite[Theorem 4.3.2]{Gigli} . 

\begin{lemma}\label{gradient_flow}
    Let $h:\mathbb{R}^n\rightarrow\mathbb{R}$ be a $C^2-$smooth convex function. Suppose that $\vec{\xi}:[0,\infty)\rightarrow\mathbb{R}^n$ is a gradient flow of $h.$ Then for any $\tau>0$ and $\vec{\xi}^*\in\mathbb{R}^n$, we have
    \begin{align*}
        \left\lvert \nabla h(\vec{\xi}(\tau)) \right\rvert^2 \le \left\lvert \nabla h(\vec{\xi}^*) \right\rvert^2+\frac{1}{\tau^2} \left\lvert \vec{\xi}^*-\vec{\xi}(0)\right\rvert^2.
    \end{align*}
\end{lemma}

With the help of Theorem \ref{Main_result_1}, Proposition \ref{prop:strata} and Lemma \ref{gradient_flow}, we can prove the convergence of the combinatorial Ricci flow.

\begin{proof}[Proof of Theorem \ref{Main_result_3}]
    By Lemma \ref{solu}, there exists a unique solution $\vec{S}(t)$ to the flow \eqref{CF2} with initial value $\vec{S}(0):=\ln \vec{k}(0) \in \partial^J\Omega_k$ and $t\in[0,\infty)$. 
    Let $\vec{k}(t)$ be the preimage of $\vec{S}(t)$ by the diffeomorphism $\ln: \partial^J\Omega_k \to \partial^J\Omega_S$. 
    Then $\vec{k}(t)$ is a solution of \eqref{CRF1} with initial value $\vec{k}(0)$. 
    By Proposition \ref{prop:strata}, let $\vec{k}^*:= \widehat{\Phi}^{-1}(\vec{T}) \in \partial^I\Omega_k$.

 
 It remains to show that $\vec{k}(t)\rightarrow \vec{k}^*$ as $t\to +\infty$.
 We define an auxiliary curve $\vec{k}^A(t)\ (t \geq 0)$ in $\partial^J\Omega_k$ as 
$$
\vec{k}_v^A(t):=
\begin{cases}
1/t, & \text{if } v \in I \backslash J, \\
k_v^*, & \text{if } v \notin I.
\end{cases}
$$
 By Theorem \ref{Main_result_1}, we have 
 \begin{align}\label{limit_flow}
    \lim_{t \to \infty} T_v(\vec{k}^A(t)) =  T_v
 \end{align}
 holds for all $v \in V$. 
 Let $\vec{S}^A(t) := \ln \vec{k}^A(t)$ be the image of this curve in $\Omega^J_S$ under the change of variables.
 By Lemma \ref{gradient_flow} and the fact that $\vec{S}(t)$ is the gradient flow of $\widetilde{\mathcal{E}}^I$, we have 
 \begin{equation}\label{eq:limite_curvature}
 \left\lvert \nabla \widetilde{\mathcal{E}}^I(\vec{S}(t)) \right\rvert^2 
 \le \left\lvert \nabla \widetilde{\mathcal{E}}^I(\vec{S}^A(t)) \right\rvert^2 + \frac{1}{t^2} \left\lvert\vec{S}^A(t)-\vec{S}(0)\right\rvert^2,\quad \forall t>0.
 \end{equation}
Consider the convergence of RHS of \eqref{eq:limite_curvature}. 
Recall \eqref{eq:deriv_engy} that 
\begin{equation*}
    \left\lvert \nabla \widetilde{\mathcal{E}}^I(\vec{S}(t)) \right\rvert^2 = \sum_{v \in V} (T_v(\vec{k}(t)) - T_v)^2 
    \quad \text{and} \quad 
    \left\lvert \nabla \widetilde{\mathcal{E}}^I(\vec{S}^A (t)) \right\rvert^2 = \sum_{v \in V} (T_v(\vec{k}^A(t)) - T_v)^2 .
\end{equation*}
By the convergence \eqref{limit_flow}, we have the first term of RHS of \eqref{eq:limite_curvature} tends to $0$ as $t$ tends to infinity. 
The convergence of the second term of RHS of \eqref{eq:limite_curvature} follows from the observation that 
\begin{align*}
    \left\lvert \vec{S}^A(t)-\vec{S}(0) \right\rvert^2 
    &= O (|\ln t|^2) .
\end{align*}
By applying convergences above to \eqref{eq:limite_curvature}, we have
$$\sum_{v \in V} (T_v(\vec{k}(t)) - T_v)^2 = \left\lvert \nabla \widetilde{\mathcal{E}}^I(\vec{S}(t)) \right\rvert^2$$
tends to $0$ as $t$ tends to infinity, 
that is, 
 \begin{align*}
     \lim_{t\rightarrow\infty}\widehat{\Phi}(\vec{k}(t))=\vec{T}\in\partial_0^I\Omega_T .
 \end{align*}
 Combining Theorem \ref{Main_result_1}, we finish the proof. 
\end{proof}

\section{Acknowledgments}
All the authors sincerely express their gratitude to Xin Nie for his abundant and generous communication on the topic of circle packing and related fields. 

Guangming Hu is supported by NSFC  (No. 12101275) and Natural Science Research Start-up Foundation of Recruiting Talents of Nanjing University of Posts and Telecommunications (No. NY224040). 
Sicheng Lu is supported by China Postdoctoral Science Foundation (Certificate Number 2024M753071).
Dong Tan is supported by the NSFC (No. 12001122, 12271533, 12361014) and the special grants for Guangxi Ba Gui Scholars.
Youliang Zhong is partially supported by NSFC (No. 12271174) and
Guangdong Basic and Applied Basic Research Foundation (No. 2023A1515010929), 
and Guangzhou Science and Technology Program (No. 202201010160).
Puchun Zhou is supported by Shanghai Science and Technology Program (Project No. 22JC1400100).

\begin{appendices}
\newcommand{\inner}[2]{\langle #1, #2 \rangle}
\newcommand{\bvec}{\boldsymbol}
\section{Hyperbolic calculations}\label{sec:app}
We prove Theorem \ref{thm:T123}, Corollary \ref{coro:k_P} and part of Lemma \ref{lem:closed_form} in this appendix.

\subsection{The hyperboloid and Klein model of hyperbolic plane}\label{AppSec:model}
We first recall some basic concepts and formulae for hyperbolic plane. See \cite[Chapter 3]{Rat94}, \cite[Chapter 2]{thurston} for more details.

The \DEF{Lorentzian 3-space} $\Lrz$ is the vector space $\mathbb{R}^3$ endowed with the Lorentzian inner product 
\begin{equation}
    \inner{\bvec{x}}{\bvec{y}} :=x_1 y_1 + x_2 y_2 - x_3 y_3.
\end{equation} 
A vector $ \bvec{x} \in\Lrz$ is said to be \DEF{time-like, space-like and light-like} if $\inner{\bvec{x}}{\bvec{x}}$ is negative, positive and zero. The \DEF{light cone} is the set of all light-like vectors in $\Lrz$.

The hyperbolic plane is the subset
\[ \mathbb{H}:=\big\{\ \bvec{x} \in\Lrz\ \lvert\ \inner{\bvec{x}}{\bvec{x}} = -1,\ x_3>0\ \big\}. \]
The Lorentzian inner product induces a Riemannian metric on $\mathbb{H}$. Geodesics in $\mathbb{H}$ are the intersection of $\mathbb{H}$ with planes through the origin of $\Lrz$. 

Every time-like $ \bvec{x}\in\Lrz$ can be normalized as $\bar{\bvec{x}}:=\bvec{x}/\sqrt{-\inner{\bvec{x}}{\bvec{x}}} \in \mathbb{H}$. 
Every light-like $\bvec{x}\in\Lrz$ can be normalized as $\bar{\bvec{x}}:=(x_1,x_2,1)$ with $x_1^2+x_2^2=1$. 
And every space-like $\bvec{x}\in\Lrz$ can be normalized as $\bar{\bvec{x}}:=\bvec{x}/\sqrt{\inner{\bvec{x}}{\bvec{x}}}$. It is also associated with a \DEF{dual geodesic} $\bvec{x}^\perp:=\{\ \bvec{y} \in\mathbb{H}\ \lvert\ \inner{\bvec{x}}{\bvec{y}} =0\ \}$. In this case, every point $\bvec{y}\in \bvec{x}^\perp$ determines a geodesic lying on the plane through $\bvec{x},\bvec{y}$ and the origin. This geodesic is always perpendicular to $\bvec{x}^\perp$.

The hyperbolic distance is strongly related to the Lorentzian inner product. This induces a simple and uniform expression for generalized circles.
\begin{proposition}\label{prop:dist_and_inner} Let $\bvec{x},\bvec{y}$ be two different points in $\Lrz$. ~
\begin{enumerate}
    \item For time-like $\bvec{x},\bvec{y}$, the hyperbolic distance between the normalized vectors $\bar{\bvec{x}},\bar{\bvec{y}}\in\mathbb{H}$ 
    is given by
    \[ \cosh d(\bar{\bvec{x}},\bar{\bvec{y}}) = - \frac{ \inner{\bvec{x}}{\bvec{y}} }{ \sqrt{ \inner{\bvec{x}}{\bvec{x}} \inner{\bvec{y}}{\bvec{y}}} }. \]
    \item For time-like $\bvec{x}$ and space-like $\bvec{y}$, the distance between $\bar{\bvec{x}}\in\mathbb{H}$ and the dual geodesic $\bvec{y}^\perp$ is given by
    \[ \sinh d(\bar{\bvec{x}},\bvec{y}^\perp) = \pm \frac{ \inner{\bvec{x}}{\bvec{y}} }{ \sqrt{ - \inner{\bvec{x}}{\bvec{x}} \inner{\bvec{y}}{\bvec{y}} } }. \]
    The signature is negative if and only if the plane through $\bvec{y}^\perp$ and the origin separates $\bvec{x},\bvec{y}$. 
    \item Let $\bvec{x},\bvec{y}$ be space-like.
    \begin{itemize}
        \item If the plane through $\bvec{x},\bvec{y}$ and the origin intersects $\mathbb{H}$, then the two dual geodesics $\bvec{x}^\perp, \bvec{y}^\perp$ are disjoint and the distance between them is given by
        \[ \cosh d(\bvec{x}^\perp, \bvec{y}^\perp) = \pm \frac{ \inner{\bvec{x}}{\bvec{y}}  }{ \sqrt{ \inner{\bvec{x}}{\bvec{x}}  \inner{\bvec{y}}{\bvec{y}} } }. \]
        The choice of signature is similar to the previous case.
        \item If the plane through $\bvec{x},\bvec{y}$ and the origin is disjoint from $\mathbb{H}$, then the two dual geodesics $\bvec{x}^\perp, \bvec{y}^\perp$ intersect and the intersection angle between them is given by
        \[ \cos \theta(\bvec{x}^\perp, \bvec{y}^\perp) = \pm \frac{ \inner{\bvec{x}}{\bvec{y}}  }{ \sqrt{ \inner{\bvec{x}}{\bvec{x}}  \inner{\bvec{y}}{\bvec{y}}  } }. \]
        \item If the plane through $\bvec{x},\bvec{y}$ and the origin is tangent to the light cone, then the two dual geodesics $\bvec{x}^\perp, \bvec{y}^\perp$ intersect at the light cone and $ \inner{\bvec{x}}{\bvec{y}} / \sqrt{ \inner{\bvec{x}}{\bvec{x}} \inner{\bvec{y}}{\bvec{y}} } = \pm 1$.
    \end{itemize}
    \item A horocycle in $\mathbb{H}$ is given by
    \[ C(\bvec{v},t):=\{\ \bvec{x}\in\mathbb{H}\ \lvert\ \inner{\bvec{x}}{\bvec{v}} = t < 0 \ \}, \]
    where $\bvec{v}$ is a light-like vector normalized as $v_3=1$ and $t$ is a negative constant. 
\end{enumerate}
\end{proposition}

\begin{lemma}\label{lem:gnr_circle}
Every generalized circle is a set 
\[ C(\bvec{c},R) := \left\{\ \bvec{x}\in\mathbb{H}\ \big\lvert\ \inner{\bvec{x}}{\bvec{c}}=R\ \right\} \]
determined by a vector $\bvec{c}\in\Lrz$ and a real number $R\in\mathbb{R}$.
We call $C(\bvec{c},R)$ a generalized circle centered at $\bvec{c}$.
\end{lemma}
\begin{proof}
    A circle is the set of points that are equidistant from the center point, while a hypercycle is defined as the set of points equidistant from a center geodesic. This lemma is simply a rephrasing of origin definitions, using results in Proposition \ref{prop:dist_and_inner}.
\end{proof}

\newcommand{\DDK}{\mathbb{D}_K}
\medskip
The normalized vector in $\Lrz$ behaves discontinuously when switching from time- or space-like to light-like.
We shall pass to the Klein model for a solution and a more consistent computation. 
Let $K:=\{\ \bvec{x}\in\Lrz\ \lvert\ x_3=1\ \}$. The \DEF{Klein model} for hyperbolic plane is obtained by mapping $\mathbb{H}$ to $K$ through a projection centered at the origin of $\Lrz$. Then $\mathbb{H}$ is mapped to the unit disk $\DDK:=\{\ (x_1,x_2,1)\ \lvert\ x_1^2+x_2^2<1\ \}$ of $K$, endowed with the push-forward metric. 
Essentially the whole space $\Lrz$ projects to a real projective plane. However, under some suitable settings, the plane $K$ is enough for our usage. 

Geodesics in $\DDK$ are simply straight lines inside the disk. And every point $\bvec{x}$ outside the closed unit disk gives rise to a dual geodesic $\bvec{x}^\perp$ as before. This geodesic is the line that connects the points of tangency of the two tangent lines to $\DDK$ at point $\bvec{x}$. Formulae in Proposition \ref{prop:dist_and_inner} are still available. And the center of a generalized circle in Lemma \ref{lem:gnr_circle} can be chosen anywhere on $K$. 

\begin{figure}[t]
    \centering
    \includegraphics[width=1\linewidth]{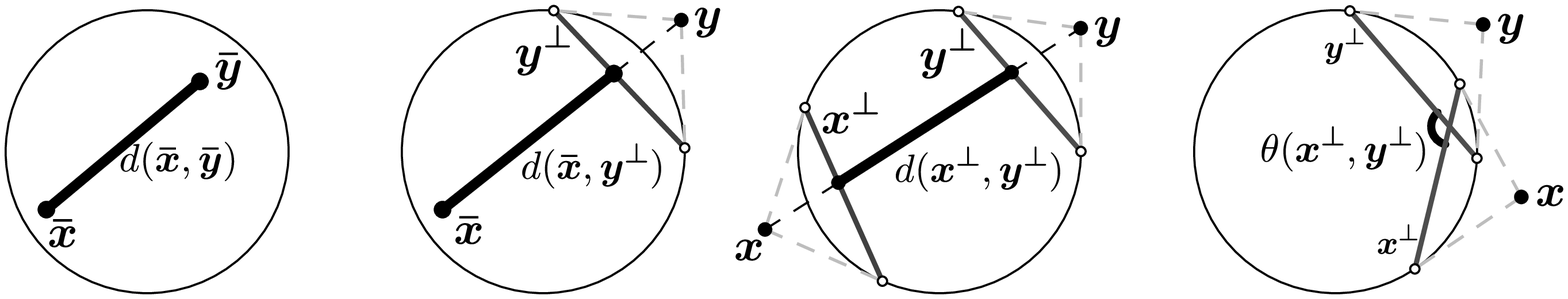}
    \caption{\small The relations between hyperbolic distance and Lorentzian inner product in various case of Proposition \ref{prop:dist_and_inner}, shown in Klein model.}
    \label{fig:inner_product}
\end{figure}

\begin{lemma}\label{lem:geod_curvt}
    The geodesic curvature of the generalized circle centered at $\bvec{c}=(c_1,c_2,1)\in K$ and passing through $\bvec{z}=(z_1,z_2,1) \in \DDK$, with $\inner{\bvec{c}}{\bvec{z}}<0$, is
    \begin{equation}\label{Appeq:geod_curvt}
        k = \frac{ -\inner{\bvec{c}}{\bvec{z}} }{ \sqrt{\inner{\bvec{c}}{\bvec{z}}^2 - \inner{\bvec{c}}{\bvec{c}}\inner{\bvec{z}}{\bvec{z}} } } \ .
    \end{equation}
    The formula still holds when $\bvec{c}$ is light-like.
\end{lemma}

\begin{proof}
    For circles and hypercycles, the formula is obtained by Proposition \ref{prop:dist_and_inner} and the relations between radius and curvature in Table \ref{table:relation}. For horocycles, the center $\bvec{c}$ is light-like and $\inner{\bvec{c}}{\bvec{c}}=0$. By the assumption $\inner{\bvec{c}}{\bvec{z}}<0$, the right-hand-side of \eqref{Appeq:geod_curvt} is $1$. 
\end{proof}

\begin{proposition}\label{prop:arc_by_kd}
      Let $\Gamma$ be an arc of a generalized circle of geodesic curvature $k$, with end points $\bvec{x},\bvec{y}\in\DDK$. Let $d=d(\bvec{x},\bvec{y})>0$ and $\theta$ be the generalized inner angle of $\Gamma$. We require $\theta\leq\pi$ when $k>1$. 
      Then the length $l$ and the total geodesic curvature $T$ of  $\Gamma$ is given by
\begin{equation}\label{Appeq:lT_in_theta}
    l(k,d)=\frac{\theta}{\sqrt{\lvert 1-k^2 \rvert}},\quad
    T(k,d)=\frac{k \theta}{\sqrt{\lvert 1-k^2 \rvert}}
\end{equation}
for all $k>0, k\neq1$ and $d>0$, where
\begin{equation}\label{Appeq:theta_in_kd}
    1+(1-k^2)(\cosh d -1) = \left\{
    \begin{array}{ll}
        \cos\theta &,\ \textrm{if}\ k>1  \\
        \cosh\theta &,\ \textrm{if}\ k<1 
    \end{array} \right.\ .
\end{equation}
Furthermore, the formula extends over $k=1$ with continuous first order derivative. In particular, when $k=1$ and $0<d<+\infty$, 
\begin{equation}\label{Appeq:lT_k=1}
    l(1,d)=T(1,d)=\sqrt{2(\cosh d -1)}\ .
\end{equation}
Hence both $l(k,d)$ and $T(k,d)$ are $C^1$-functions over $\mathbb{R}_{>0}^2$.
\end{proposition}
\begin{proof}
    (\ref{Appeq:lT_in_theta}) is simply the definition. Also see Table \ref{table:relation}.

    When $k>1$, let $\bvec{c}$ be the center of the circle. 
    By hyperbolic cosine law in the triangle with vertices $\bvec{x},\bvec{y},\bvec{c}$, 
    $ \cosh d = \cosh^2 r - \sinh^2 r \cos \theta \ . $
    Then (\ref{Appeq:theta_in_kd}) for the $k>1$ case is obtained by direct substituting.

    When $k<1$, the arc $\Gamma$, the center geodesic and the two radii at $\bvec{x},\bvec{y}$ bound a quadrilateral with 2 right angles. Then we have 
    $ \cosh d = \cosh^2 r \cosh\theta - \sinh^2 r \ . $
    The formula is obtained again by direct substituting.

    When $k=1$, \eqref{Appeq:lT_k=1} can be obtained by passing to the upper half-plane model of hyperbolic plane, with the horocycle mapped to a horizontal line.

    \medskip
    To prove the continuity, we need the following Taylor series when $x \to 0^+$:
    \begin{equation}\label{Appeq:expansion}\begin{aligned}
        \arccos(1-x) &= \sqrt{2x} + \frac{x^{3/2}}{6\sqrt2} + o(x^2) \ , \\
        \mathrm{arccosh}(1+x) &= \sqrt{2x} - \frac{x^{3/2}}{6\sqrt2} + o(x^2) \ .
    \end{aligned}\end{equation}
    
    For simplicity and future application, let 
    \begin{equation}\label{Appeq:substitute}
        D:=\cosh d -1 \quad \textrm{and} \quad K:= 1-k^2 \ .
    \end{equation} 
    Then $D'(d) = \sinh d = \sqrt{D(D+2)}$.
    
    Whenever $d$ converges in $\mathbb{R}_{>0}$, $D$ also converges in $\mathbb{R}_{>0}$ and we have 
    \begin{align*}
        \lim_{k\to 1^+} \frac{\arccos\big( 1+KD \big)}{\sqrt{-K}}  &= \lim_{K\to 0^-} \frac{\sqrt{ -2KD }}{\sqrt{-K}}  = \sqrt{ 2D }, \\
        \lim_{k\to 1^-} \frac{\arccos\big( 1+KD \big)}{\sqrt{K}}  &= \lim_{K\to 0^+} \frac{\sqrt{ 2KD }}{\sqrt{K}}  = \sqrt{ 2D }.
    \end{align*}
    Thus $l(k,d)$ is continuous on $\mathbb{R}_{>0}^2$, and so does $T(k,d)$. 

    \bigskip
    It can be shown by direct computation that when $k\neq1$,
    \begin{equation*}
        \pp{l}{d} = \sqrt{\frac{2+D}{2+KD}}
    \end{equation*}
    and it extends continuously over $k=1$. The following would also be useful: 
    \begin{equation}
        \pp{l}{D} = \frac{1}{\sqrt{D(2+KD)}} \ .
    \end{equation}

    The continuity of $\partial l / \partial k$ requires detailed computation. When $k>1$, 
    $$ \pp{l}{k} = \frac{ 2kD }{ -K\sqrt{D(2+KD)} } - \frac{ k \arccos(1+KD) }{ -K\sqrt{-K} } . $$
    By (\ref{Appeq:expansion}),
    \begin{align*}
        \lim_{k\to 1^+} \pp{l}{k}
        &= \lim_{K\to 0^-} \left( \frac{2\sqrt{D}}{-K\sqrt{2+KD}} - \frac{\sqrt{-2KD}+(-KD)^{3/2}/6\sqrt{2}}{-K\sqrt{-K}} \right) \\
        &= \lim_{K\to 0^-} \frac{\sqrt{2D}}{-K} \left( \frac{\sqrt{2}}{\sqrt{2+KD}} -1 \right) - \frac{D\sqrt{D}}{6\sqrt{2}} 
        = \frac{D\sqrt{D}}{3\sqrt{2}} \ .
    \end{align*}

    When $k<1$,
    $$ \pp{l}{k} = \frac{k\ \mathrm{arccosh}(1+KD)}{K\sqrt{K}} - \frac{2kD}{K\sqrt{D(2+KD)}} \ . $$
    By (\ref{Appeq:expansion}),
    \begin{align*}
        \lim_{k\to 1^-} \pp{l}{k}
        &= \lim_{K\to 0^+} \left( \frac{\sqrt{2KD}-(KD)^{3/2}/6\sqrt{2}}{K\sqrt{K}} - \frac{2\sqrt{D}}{K\sqrt{2+KD}} \right) \\
        &= \lim_{K\to 0^+} \frac{\sqrt{2D}}{K} \left( 1- \frac{\sqrt{2}}{\sqrt{2+KD}} \right) - \frac{D\sqrt{D}}{6\sqrt{2}} 
        = \frac{D\sqrt{D}}{3\sqrt{2}} \ .
    \end{align*}
    Hence $\pp{l}{k}$ is continuous at $k=1$. These show that $l=l(k,d)$ is a $C^1$-function on $\mathbb{R}_{>0}^2$, and so does $T=T(k,d)=k\cdot l(k,d)$. 
\end{proof}

\subsection{Three-circle configuration}\label{AppSec:2}
Now we turn to the three-circle configuration and the circle packed triangle, using the hyperboloid and Klein model this time.

We start with the hyperboloid model. Suppose in a three-circle configuration, for $i=1,2,3$, the circle $C_i$ is given by $C(\bvec{c_i},R_i)$ as Lemma \ref{lem:gnr_circle}, $\bvec{c_i}\in\Lrz,\ R_i\in\mathbb{R}$.
With the help of dual geodesic, in the definition of (possibly degenerated) generalized circle packed triangle, when some $C_i$ is actually a hypercycle (or a geodesic), the vertex $\bvec{v_i}$ is a segment on the dual geodesic $\bvec{c_i}^\perp$. An edge $e_{ij}$ is the geodesic segment realizing the inner product $\inner{ \bvec{\bar{c}_i} }{ \bvec{\bar{c}_i} }$ as Figure \ref{fig:inner_product}. This segment may have infinite or zero length. 

\bigskip
Before passing to the Klein model, we need some normalization. Since the dual triangle always has non-empty interior (even for degenerated ones in Section \ref{3}), we can apply a hyperbolic isometry so that $(0,0,1)$ lies in the interior of dual triangle. Under such setting, the boundary of circle packed triangle never passes $(0,0,1)$ and all centers $\bvec{c_i}$ will not lie on the horizontal plane $\{x_3=0\}\subset\Lrz$. Hence everything can be projected to the plane $K$, and the origin $(0,0,1)$ always lies in the interior of a triangle. We will stick to this normalization for circle packed triangles form now on.

Now let $\bvec{x_i}$ be the projection of $\bvec{c_i}$ in $K$. 
Since the three circles are externally tangent to each other, and with the normalization condition, we have $\inner{\bvec{x_i}}{\bvec{x_j}}<0$ for any $i \neq j$. 
The vertices of the dual triangle $\bvec{t_{12}}, \bvec{t_{23}}, \bvec{t_{31}}$ must lie in $\DDK$. 
Let $(i,j,k)$ be a cyclic re-ordering of $(1,2,3)$. By our setting, $\bvec{t_{ij}}$ lies inside the straight line segment connecting $\bvec{x_i}$ and $\bvec{x_j}$ , hence 
$$ \exists \ \lambda_k^i, \lambda_k^j \in (0,1)\ \textrm{with}\ \lambda_k^i+\lambda_k^j =1,\ \textrm{such that}\ \bvec{t_{ij}} = \lambda_k^i \bvec{x_i} + \lambda_k^j \bvec{x_j} $$ 
as vectors in $K\subset\Lrz$. See Figure \ref{fig:3circle_Klein} for the settings. 
We will compute required quantities by studying the relations between these coefficients. 

Let $c_{ij}:=\inner{\bvec{x_i}}{\bvec{x_j}}$ for all $i,j\in\{1,2,3\}$, and define $\delta_{ij}:=c_{ij}^2-c_{ii}c_{jj}$ when $i\neq j$.
Then when $i\neq j$, $c_{ij}<0$ by assumption. By the condition that the straight line segment connecting $\bvec{x_i}$ and $\bvec{x_j}$ intersects $\DDK$, it can be shown that $\delta_{ij}>0$.

\begin{figure}[t]
    \centering
    \includegraphics[width=0.8\linewidth]{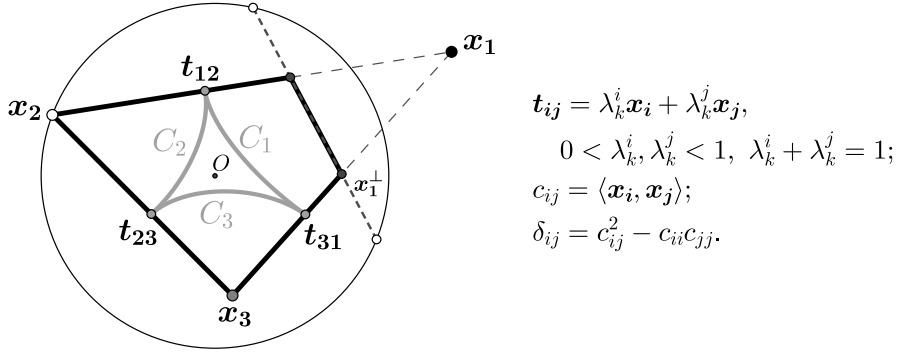}
    \caption{\small The same circle packed triangle as Figure \ref{fig:general3circle} but viewed in the Klein model. Note that circles are general conics in Klein model. The origin of $\mathbb{D}_K$ is assumed to be contained in the dual triangle. All used notations are gathered on the right.}
    \label{fig:3circle_Klein}
\end{figure}

\bigskip
Before the computation, note the symmetry of the settings of three-circle configuration. Any formula stated for a cyclic re-ordering of $(1,2,3)$ would still hold if any two indices are exchanged. This is called \DEF{the symmetry principle} for short. 

It is remarkable that all the settings above and computations below are available even when some $k_i=1$. That is the core reason for choosing the Klein model and using a proof in the flavor of projective geometry. 

    \begin{proof}[Proof of Theorem \ref{thm:T123}]
        We would use Proposition \ref{prop:arc_by_kd} to convert the computation of $T_i$ into the computation of distance between two points of tangency $\bvec{t_{ki}}, \bvec{t_{ij}}$. 
        
        Since $\bvec{t_{ki}}, \bvec{t_{ij}}\in\DDK$, by Proposition \ref{prop:dist_and_inner}.1, 
        $ \cosh d(\bvec{t_{ki}}, \bvec{t_{ij}}) = -\frac { \inner{\bvec{t_{ki}}}{\bvec{t_{ij}}} }{\sqrt{ \inner{\bvec{t_{ki}}}{\bvec{t_{ki}}} \inner{\bvec{t_{ij}}}{\bvec{t_{ij}}} }} $.
        So we need to compute $\inner{\bvec{t_{ki}}}{\bvec{t_{ij}}}$ and $\inner{\bvec{t_{ki}}}{\bvec{t_{ki}}}$.

        It can be shown by a direct computation that 
        $\inner{\bvec{x_i}}{\bvec{t_{ki}}}^2-\inner{\bvec{x_i}}{\bvec{x_i}}\inner{\bvec{t_{ki}}}{\bvec{t_{ki}}} = (\lambda_j^k)^2 \delta_{ik}$. Since $\bvec{t_{ki}}$ lies on the circle centered at $\bvec{x_i}$, by Lemma \ref{lem:geod_curvt}, we have 
        \begin{equation}\label{Appeq:xitj}
        \inner{\bvec{x_i}}{\bvec{t_{ki}}} = -\lambda_j^k k_i \sqrt{\delta_{ik}}\ .
        \end{equation}
        Applying this formula and the symmetry principle, we have
        \begin{equation}\label{Appeq:tjtj}
        \inner{\bvec{t_{ki}}}{\bvec{t_{ki}}} 
        = \lambda_j^i \inner{\bvec{x_i}}{\bvec{t_{ki}}} + \lambda_j^k \inner{\bvec{x_k}}{\bvec{t_{ki}}}
        = -\lambda_j^i \lambda_j^k \sqrt{\delta_{ik}} (k_i+k_k) \ .
        \end{equation}

        Note that $\inner{\bvec{x_i}}{\bvec{t_{ki}}} = \lambda_j^i c_{ii} + \lambda_j^k c_{ik}$. 
        Then (\ref{Appeq:xitj}) implies 
        $\lambda_j^k (k_i \sqrt{\delta_{ik}} + c_{ik}) = -\lambda_j^i c_{ii}$. 
        Since $\lambda_j^i \lambda_j^k \neq 0$, by the symmetry principle, we have 
        $ (c_{ik} + k_i \sqrt{\delta_{ik}}) (c_{ik} + k_k \sqrt{\delta_{ik}}) = c_{ii} c_{kk} $.
        Substituting $\delta_{ik} = c_{ik}^2 - c_{ii} c_{kk}$ into the above equation and simplifying, we obtain
        \begin{equation}\label{Appeq:cij}
        -c_{ik} = \frac{1+k_i k_k}{k_i+k_k} \sqrt{\delta_{ik}}\ .
        \end{equation}

        On the other hand, both $\bvec{t_{ki}}, \bvec{t_{ij}}$ lie on the circle centered at $\bvec{x_i}$. By Lemma \ref{lem:gnr_circle}, 
        $ \frac{\inner{\bvec{x_i}}{\bvec{t_{ki}}}^2}{\inner{\bvec{t_{ki}}}{\bvec{t_{ki}}}} = \frac{\inner{\bvec{x_i}}{\bvec{t_{ij}}}^2}{\inner{\bvec{t_{ij}}}{\bvec{t_{ij}}}} $.
        A direct computation gives 
        $ \frac{\inner{\bvec{x_i}}{\bvec{t_{ki}}}^2}{\inner{\bvec{t_{ki}}}{\bvec{t_{ki}}}} = 
        \frac{(\lambda_j^k)^2\delta_{ik}}{\inner{\bvec{t_{ki}}}{\bvec{t_{ki}}}} - c_{ii} $.
        Hence we have
        $ \frac{(\lambda_j^k)^2\delta_{ik}}{\inner{\bvec{t_{ki}}}{\bvec{t_{ki}}}} = \frac{(\lambda_k^j)^2\delta_{ij}}{\inner{\bvec{t_{ij}}}{\bvec{t_{ij}}}} $.
        Using (\ref{Appeq:tjtj}), we get
        \begin{equation}\label{Appeq:sqrdelta}
        \frac{\lambda_k^j\lambda_j^i \sqrt{\delta_{ij}}}{k_i+k_j} =
        \frac{\lambda_j^k\lambda_k^i \sqrt{\delta_{ik}}}{k_i+k_k} \ .
        \end{equation}    
        Now that all quantities in the above are non-zero, by a cyclic product of the formula, we get the following Ceva type formula:
        \begin{equation}\label{Appeq:ceva}
        \lambda_i^j \lambda_j^k \lambda_k^i = \lambda_i^k \lambda_k^j \lambda_j^i \ .
        \end{equation}

        Finally we are able to compute $\cosh d(\bvec{t_{ki}},\bvec{t_{ij}})$.
        By (\ref{Appeq:tjtj}) and (\ref{Appeq:sqrdelta}), 
        \begin{align*}
        \inner{\bvec{t_{ki}}}{\bvec{t_{ki}}}\inner{\bvec{t_{ij}}}{\bvec{t_{ij}}} &= \lambda_j^k \lambda_j^i \sqrt{\delta_{ik}} (k_i+k_k) \ \lambda_k^i \lambda_k^j \sqrt{\delta_{ij}} (k_i+k_j) \\
        &= [\lambda_j^k \lambda_k^i \sqrt{\delta_{ik}} (k_i+k_j)]^2\ .
        \end{align*}
        And
        \begin{align*}
        \inner{\bvec{t_{ij}}}{\bvec{t_{ki}}} 
        &= \lambda_k^i \inner{\bvec{x_i}}{\bvec{t_{ki}}} + \lambda_k^j \lambda_j^i c_{ij} + \lambda_k^j \lambda_j^k c_{jk} \\
        &= -\lambda_k^i \lambda_j^k k_i \sqrt{\delta_{ik}} - \lambda_k^j \lambda_j^i \frac{1+k_i k_j}{k_i+k_j}\sqrt{\delta_{ij}} + \lambda_k^j \lambda_j^k c_{jk} \quad (\textrm{By (\ref{Appeq:xitj}) and (\ref{Appeq:cij})}) \\ 
        &= -\lambda_k^i \lambda_j^k k_i \sqrt{\delta_{ik}} - \frac{\lambda_j^k \lambda_k^i \sqrt{\delta_{ik}}}{k_i+k_k} (1+k_i k_j) + \lambda_k^j \lambda_j^k c_{jk} \quad (\textrm{By (\ref{Appeq:sqrdelta})}) \\ 
        &= -\lambda_j^k \lambda_k^i \sqrt{\delta_{ik}} \left( k_i + \frac{1+k_i k_j}{k_i+k_k} \right) + \lambda_k^j \lambda_j^k c_{jk} \ . 
        \end{align*}
        
        Hence
        $$ \frac { \inner{\bvec{t_{ki}}}{\bvec{t_{ij}}} } {\sqrt{ \inner{\bvec{t_{ki}}}{\bvec{t_{ki}}} \inner{\bvec{t_{ij}}}{\bvec{t_{ij}}} } } =
        -\frac{k_i + \frac{1+k_i k_j}{k_i+k_k}}{k_i+k_j} 
        +\frac{\lambda_k^j c_{jk}}{\lambda_k^i \sqrt{\delta_{ik}} (k_i+k_j)} \ .$$
        By (\ref{Appeq:ceva}) and (\ref{Appeq:sqrdelta}),
        $$ \frac{\lambda_k^j}{\lambda_k^i} = \frac{\lambda_i^j \lambda_j^k}{\lambda_j^i\lambda_i^k} = \frac{\sqrt{\delta_{ik}}}{\sqrt{\delta_{jk}}} \frac{k_k+k_j}{k_k+k_i}\ . $$
        Together with (\ref{Appeq:cij}),
        $$ \frac{\lambda_k^j c_{jk}}{\lambda_k^i \sqrt{\delta_{ik}} (k_i+k_j)} = \frac{c_{jk}}{(k_i+k_j)} \frac1{\sqrt{\delta_{jk}}} \frac{k_k+k_j}{k_k+k_i} = \frac{-(1+k_j k_k)}{(k_i+k_j)(k_i+k_k)} \ .$$
        Then we have
        \begin{equation}\label{Appeq:dist_tjtk}
        \cosh d(\bvec{t_{ki}},\bvec{t_{ij}}) = - \frac { \inner{\bvec{t_{ki}}}{\bvec{t_{ij}}} } {\sqrt{ \inner{\bvec{t_{ki}}}{\bvec{t_{ki}}} \inner{\bvec{t_{ij}}}{\bvec{t_{ij}}} } } = 
        1 + \frac{2}{(k_i+k_j)(k_i+k_k)} \ .
        \end{equation}
        The desired formula (\ref{eq:k123}) 
        is obtained from the above and Proposition \ref{prop:arc_by_kd}.

        \bigskip
        The smoothness is obtained by the chain rules. Consider $(i,j,k)=(1,2,3)$ as an example. Let $l$ be the function of two variables in Proposition \ref{prop:arc_by_kd}, and $D$ be defined as \eqref{Appeq:substitute}. Now $D$ is also a function of all $k_i$'s, so we have 
        \[ \pp{l_1}{k_1}(k_1,k_2,k_3) = \pp{l}{k}(k_1, D_1) + \pp{l}{D}(k_1, D_1) \cdot \pp{D_1}{k_1}(k_1,k_2,k_3), \]
        where $D_1=2/(k_1+k_2)(k_1+k_3)$ from \eqref{Appeq:dist_tjtk}.
        Similarly, 
        \[ \pp{l_1}{k_2}(k_1,k_2,k_3) = \pp{l}{k}(k_1, D_1) \cdot \pp{D_1}{k_2}(k_1,k_2,k_3) . \]
        We have shown the continuity of $\pp{l}{k}$ and $\pp{l}{D}$ in the proof of Proposition \ref{prop:arc_by_kd}, and $D_1$ is a smooth function on $\mathbb{R}^3$ by \eqref{Appeq:dist_tjtk}. Hence, $l_1$ is a $C^1$-function of all $k_i$'s. And so does the total geodesic curvature $T_1=k_1 \cdot l_1$.

    \end{proof}

\begin{proof}[Proof of Corollary \ref{coro:k_P}]
    We shall use the relations in the bigon intersected by two orthogonal circles. See \cite[Lemma 2.7]{BHS}.

    When $k_1>1$, the relation between the inner angle $\theta_1$ and curvature $k_1, k_P$ is given by
    $$ \cot\frac{\theta_1}{2} = \frac{k_P}{\sqrt{k_1^2-1}} \ . $$
    By \eqref{eq:k123}, $\cos\theta_1 = 1+ \frac{2(1-k_1^2)}{(k_1+k_2)(k_1+k_3)}$. Then the desired formula \eqref{eq:k_P} is obtained by the trigonometric identity $\cos2\theta = (\cot^2 \theta -1)/(\cot^2 \theta +1)$.

    When $k_1<1$, we have 
    $$ \coth \frac{\theta_1}{2} = \frac{k_P}{\sqrt{1-k_1^2}} $$
    and the identity $\cosh2\theta = (\coth^2 \theta +1)/(\coth^2 \theta -1)$. The computation is similar.

    When $k_1=1$, we shall consider the length of the horocycle arc and we have 
    $$ l_1 = T_1 = \frac{2}{k_P} \ .$$
    Together with the formula for $T_1$ when $k_1=1$, we have $k_P^2=(k_2+1)(k_3+1)$. This shows \eqref{eq:k_P} still holds when some of the $k_i$'s is $1$.
\end{proof}

\begin{proof}[Proof of Lemma \ref{lem:closed_form}, continued]
    By the proof of Proposition \ref{prop:arc_by_kd}, we have 
    \[ \pp{l}{D} = \frac{1}{\sqrt{D(2+KD)}} \quad \textrm{and} \quad \pp{l}{k} = \frac{k}{K}\left( l - 2\sqrt{\frac{D}{2+KD}}\right) \  .\]
    
    When $k_3=0$ and $k_1,k_2>0$, by \eqref{Appeq:dist_tjtk}, 
    \[ D_1 := \cosh d(t_{12}, t_{31}) -1 = \frac{2}{k_1(k_1+k_2)}, \quad
       D_2 := \cosh d(t_{12}, t_{23}) -1 = \frac{2}{k_2(k_1+k_2)} \]
    are always well-defined. Hence the above formulae still hold for $k_3=0$ when taking partial derivative with respect to $k_1$ or $k_2$ on $\partial^{\{3\}}\Omega_k^\Delta$. By the chain rules,  
    \[ \pp{l_1}{k_2} = \pp{l}{D}(k_1, D_1) \cdot \pp{D_1}{k_2} 
    = \frac{1}{\sqrt{D_1 \left( 2+(1-k_1^2) D_1 \right)}} \frac{-D_1}{k_1 + k_2}
    = \frac{-1}{(k_1+k_2)\sqrt{1+k_1k_2}}\ ,
    \]
    \[ \pp{l_2}{k_1} = \pp{l}{D}(k_2, D_2) \cdot \pp{D_2}{k_1} 
    = \frac{1}{\sqrt{D_2 \left( 2+(1-k_2^2) D_2 \right)}} \frac{-D_2}{k_1 + k_2}
    = \frac{-1}{(k_1+k_2)\sqrt{1+k_1k_2}}\ .
    \]
    Hence $\pp{l_1}{k_2}=\pp{l_2}{k_1}$. And we prove that $T_1\mathrm{d}S_1+ T_2\mathrm{d}S_2$ is closed.
    \end{proof}
\end{appendices}

\bigskip 

Guangming Hu

\textsc{ College of Science, Nanjing University of Posts and Telecommunications, Nanjing, 210003, P.R. China.} 
\textit{Email address}: \texttt{18810692738@163.com}

\medskip
Sicheng Lu 

\textsc{ The Institute of Geometry and Physics, University of Science and Technology of China, Hefei, 230026, P. R. China
}
\textit{Email address}: \texttt{sichenglu@ustc.edu.cn}

\medskip
Dong Tan

\textsc{Guangxi Center for Mathematical Research, Guangxi University, 530004, Nanning, China}
\textit{Email address}: \texttt{duzuizhe2013@foxmail.com}

\medskip
Youliang Zhong 

\textsc{ School of Mathematics, South China University of Technology, 510641, Guangzhou, P. R. China. }
\textit{Email address}: \texttt{zhongyl@scut.edu.cn}

\medskip
Puchun Zhou 

\textsc{ School of Mathematical Sciences, Fudan University, Shanghai, 200433, P.R. China }
\textit{Email address}: \texttt{pczhou22@m.fudan.edu.cn}

\end{document}